\documentclass[aos,draft]{imsart}

\usepackage[utf8]{inputenc}
\usepackage[T1]{fontenc}
\usepackage[english]{babel}
\usepackage{xcolor}
\usepackage{amsmath, latexsym, amssymb,amsthm,enumerate}
\usepackage{natbib}
\usepackage{tikz}
\usetikzlibrary{shapes}

\newtheorem{thm}{Theorem}[section]
\newtheorem{prop}[thm]{Proposition}
\newtheorem{cor}[thm]{Corollary}
\newtheorem{defin}[thm]{Definition}
\newtheorem{lem}[thm]{Lemma}
\newtheorem{rem}[thm]{Remark}
\newtheorem{rems}[thm]{Remarks}
\newtheorem*{assumA}{Assumption A}
\newtheorem*{assumB}{Assumption B($k$)}

\startlocaldefs
\renewcommand{\le}{\leqslant}
\renewcommand{\ge}{\geqslant}  
\newcommand{\lec}{\preccurlyeq}
\newcommand{\gec}{\succcurlyeq}

 \newcommand{\norm}[1]{\left\lVert #1 \right\rVert}  
 \newcommand{\interior}[1]{\mathring{#1}}
\newcommand\eps{\varepsilon}

\newcommand\e{\mathrm{e}}
\newcommand\dd{\mathrm{d}}

\newcommand\lb{[\![}
\newcommand\rb{]\!]}

\newcommand\T{\mathcal{T}}
\newcommand\1{\mathbf{1}}
\newcommand\Glm{\mathcal{G}_{\le m}}
\newcommand\Glmo{\mathcal{G}_{\le m_0}}
\newcommand\Gm[1]{\mathcal{G}_{m_{#1}}}
\newcommand\Gj{\mathcal{G}_{j}}
\newcommand\Gf{\mathcal{G}_{<\infty}}
\newcommand\R{\mathbb{R}}
\newcommand\E{\mathbb{E}}
\renewcommand\P{\mathbb{P}}
\newcommand\N{\mathbb{N}}

\newcommand\D[1]{\mathrm{Desc}(#1)}
\newcommand\C[1]{\mathrm{Child}(#1)}

\newcommand\parent[1]{{#1}^{\uparrow}}
\newcommand\fs{\frak{s}}
\newcommand\fe{\frak{e}}
\newcommand\fd{\frak{d}}

\renewcommand\ln{\mathrm{Log}}

\DeclareMathOperator{\Diam}{Diam}
\endlocaldefs

\begin{document}

\begin{frontmatter}

\title{Optimal rates for finite mixture estimation}
\runtitle{Optimal rates for finite mixture estimation}


\author{\fnms{Philippe} \snm{Heinrich}\ead[label=e1]{philippe.heinrich@math.univ-lille1.fr}}
\address{\printead{e1}}
\and
\author{\fnms{Jonas} \snm{Kahn}\corref{JK}\ead[label=e2]{jonas.kahn@math.univ-lille1.fr}}
\address{\printead{e2}}
\affiliation{Universit\'e Lille 1\\ Laboratoire Paul Painlev\'e B\^at. M2\\
  Cit\'e Scientifique \\59655 Villeneuve d'Ascq,  FRANCE}

\runauthor{P. Heinrich and J. Kahn}
\begin{abstract} 
    \ We study the rates of estimation of finite mixing distributions, that is, the parameters of the mixture. We prove that under some regularity and strong
  identifiability conditions, around a given mixing distribution with $m_0$ components, the optimal local minimax rate
  of  estimation  of a mixing distribution with $m$ components is $n^{-1/(4(m-m_0) +
  2)}$.  This corrects a previous paper  by   \citet{Chen} in The Annals of
Statistics.

By contrast, it turns out that there are estimators with a
(non-uniform) pointwise rate of estimation of $n^{-1/2}$ for all mixing distributions with a finite number of components.
\end{abstract}



\end{frontmatter}

\section{Introduction}
Finite mixture models go back to the work of \citet{Pear} who studied
biometrical ratios on crabs. As a flexible tool to grasp heterogeneity
in data, these models have emerged and successfully been applied in various fields including astronomy, biology, genetics, economy, social sciences and engineering.  A general introduction as well as a brief history can be found in the book of \citet{McLP}. 

\medskip

There are essentially three cases where finite mixtures and their estimation naturally arise. One actively investigated topic is model-based clustering. Here the aim is to divide the data into $k$ clusters and assign (new) data to a cluster. A possible approach is to consider that each data point from a cluster is generated according to a density probability known up to a few parameters, so that the whole data is generated by mixture with $k$ components \citep{McLP,Teh}.

The second, more traditional case, is the statistical description of possibly heterogeneous data where the underlying mixing distribution has no particular meaning. In that case, mixtures are a tool to describe efficiently the ``true'' probability distribution. The goal is then to control the convergence rate of mixture estimators  to this ``true'' probability measure \citep{VdG, MR1873329, MR1810921}.

In the third case, we are interested in the mixing distribution itself, that is the parameters of the mixture. The support points and proportions are the parameters we want to estimate. Typically, they correspond to the phenomenon that is studied, but we only observe data points distributed under the probability distribution corresponding to the mixture. This is the case we are interested in.

Notice that some works try to bridge the gap between the estimation of the mixture, and the estimation of the mixing distribution, usually at least through estimation of the number of components -- the order -- in the finite mixture. In particular, \citet{rousseau2011asymptotic}
have proved that their Bayesian estimator of the mixture tends to empty the extra components, and \citet{MR3015722} have given the minimal penalty on the maximum likelihood estimator that yields strong consistency on the order.

One could expect that a good estimator for the mixture would be a good estimator for the mixing model. However, this is not so clear. The situation is reminiscent of the difference between estimation and identification in model selection, where \citet{yang2005can} has proved that no procedure can be optimal for both. Moreover, rates of convergence can be very different, as illustrated in an infinite-dimensional case by \citet{MR3263130}.

\medskip

When the aim is to estimate the mixture parameters, optimal rates are a key information.  These were unknown \citep[see e.g.][]{Tit} till the work of \citet{Chen}, who established a $n^{-1/4}$ local minimax rate, under reasonable identifiability conditions, for one-dimensional-parameter mixtures.

This result is somewhat surprising, since the rate does not depend on
the number of components. In particular, as a rule of thumb, if a
continuous parameter (here, the rate exponent) is constant for all big
integers, it is the same in the infinite case. However, mixtures with
an infinite number of components can only be estimated at a
non-parametric rate in general. Indeed, deconvolution may be viewed as
a special case of an infinite mixture problem: estimating the mixing
distribution of the shifts of the probability measure of the noise. However, \citet{Fan} had proved that the $L^2$-convergence rate was (a power of) logarithmic in general, and \citet{CCDM} and \citet{DedMic} have generalized this kind of rates to different Wasserstein metric, including the $L^1$-Wasserstein metric. The latter is the one used by \citet{Chen}.

A possible explanation could have been a constant in front of the rate that would explode with the number of components. It turns out, however, that the result by \citet{Chen} is erroneous. 

Let us be more specific. In his Theorem~1, \citet{Chen} proves an $n^{-1/4}$ lower bound on the local minimax rate. Lemma~2 provides a control on a power $\alpha $ of the transportation distance between two mixing distributions by the $L^{\infty}$-distance between the corresponding probability distribution functions. This control is uniform on all pairs of mixing distribution in a ball around a mixing distribution $G_0$.
This uniform control entails (Theorem~2) an upper bound $n^{-1/(2 \alpha )}$ on the local minimax rate of estimation.

The exponent $\alpha$ in Lemma~2 was equal to $2$, so that the lower
and upper bounds coincide. However, Lemma~2 and its proof contain an
error: forgetting that distinct components can converge to the same
one.  Our article aims at giving correct statements and proofs for this Lemma~2 and its consequences.

The main part consists in finding the correct $\alpha$; that is Theorem~\ref{main}. Theorem~\ref{lower_bound} gives the matching lower bound, so that the local minimax rate is established.

Interestingly, another way to correct Lemma~2 is by restricting the pairs of mixtures that are compared. Namely, instead of comparing all pairs of mixtures in a ball around $G_0$, we allow only comparison of a mixture in the ball with the ball center $G_0$. Then $\alpha = 2$ is valid. We give the corresponding statement in Theorem~\ref{weak_Chen}. Translated to Theorem~2 of \citeauthor{Chen}, this corresponds  to dropping uniformity. That is, for any fixed $G$, the same estimator will converge at rate $n^{-1/4}$, but the constant depends on $G$: this is a bound on pointwise rate everywhere, instead of a bound on local minimax rate.

Thus the optimal local minimax rate and the optimal pointwise rate of estimation everywhere do not coincide. This discrepancy is not very usual in statistics, and often a source of confusion. To make things a little clearer, we also establish the optimal pointwise rate in Theorem~\ref{choose_and_estimate}. Since Theorem~1 of \citet{Chen} is a bound on local minimax rate, the pointwise rate might be better. And indeed,  the optimal pointwise rate everywhere is $n^{-1/2}$.


\medskip

The paper by \citet{Chen} has been widely cited and used. 
Apart from applied papers citing it that may have relied on the theoretical guarantees  \citep[see e.g.][]{kuhn2014spatial, liu2014unrestricted}, there are essentially two ways it could play a role. Firstly, when it is used as part of a proof, secondly when it is used as a benchmark.

The first case covers papers that generalize \citeauthor{Chen}'s result in other settings, and re-use its theorems and proofs. For example, \citet{Ish} propose a Bayesian estimator that achieves the $n^{-1/4}$ frequentist rate, and use  \citet[Lemma~2]{Chen} in their analysis. More recently, \citet{Ngu} generalizes those results to mixtures with an abstract parameter space and indefinite number of components. However his Theorem~1 generalizes  \citet[Lemma~2]{Chen} while transposing the proof with the mistake. The main results of both these articles hold however: they do not need the full strength of  \citet[Lemma~2]{Chen}, but merely the weaker version Theorem \ref{weak_Chen}.

These two papers also use \citeauthor{Chen}'s \citeyearpar{Chen} article as a benchmark. However, the optimal pointwise rate everywhere would probably be a better reference point in their case, as in many others. In particular, it seems likely that a Bayesian estimator could converge pointwise at speed $n^{-1/2}$ everywhere. We have not checked whether the proof by \citet{Ish} can be improved, or if another prior is necessary. 

This use as a benchmark is very usual, as expected for this kind of optimality result \citep[see e.g.][]{Zhu,Zhu2}. Let us point in particular to a result by \citet{Ryan}. He achieves almost $n^{-1/2}$ rate for the predictive recursion algorithm, and tries to explain the discrepancy with \citet{Chen} by the fact that the parameters are constrained to live in a finite space for his algorithm. In fact, since his rate is pointwise, it fits with the continuous case. 


\bigskip

In Section~\ref{not-res}, we give the notations and define and discuss the regularity 
assumptions we use. In Section~\ref{sec:main}, we state and discuss the main theorems, giving the optimal local minimax rate and pointwise rate everywhere. We try to give some intuition. We also dwell on the interpretation and practical consequences of having different rates, and conclude the section with open questions. In Section~\ref{key}, we give and explain the meaning of the key intermediate results and prove the main theorems from here. In Section~\ref{proofs}, we prove those key intermediate results. In particular, in Section~\ref{tree}, we introduce the most original tool of our proofs: the coarse-graining tree that allows to patch the mistake in the article by \citet{Chen}.

Some auxiliary and technical results are  detailed in appendices grouped in a supplemental
part \citep{Supp}. 

\section{Notations and regularity conditions}
\label{not-res}

\subsection{General notations}
\label{gen_not}

Throughout the paper, the family $\left\{f(x,\theta)\right\}_{\theta \in \Theta }$ will consist of probability densities $x\mapsto f(x,\theta)$ on $\R$ with respect to some $\sigma$-finite measure $\lambda$. The
parameter set $\Theta $ is always assumed to be a compact subset of
$\R$ with non-empty interior. We write $\Diam \Theta$ for its diameter. Given an $m$-mixing (or $m$-points support) distribution $G$ on $\Theta$, a finite mixture model with $m$
components is defined by 
\begin{equation}
  \label{fxG}
  f(x,G)=\int_{\Theta}f(x,\theta)\dd G(\theta).
\end{equation}
The set of such $m$-mixing distributions $G$ is denoted by $\Gm{}$ and $\Glm$
will be the union of $\Gj$ for $j\in \lb 1,m\rb$. Similarly, the set of finite mixing distributions is denoted by $\Gf$. For two mixing
distributions $G_1$ and $G_2$, note that by linearity
$f(x,G_1-G_2)=f(x,G_1)-f(x,G_2)$. This will be used to shorten
expressions. 

In what follows $\|\cdot\|_\infty$ is the supremum norm with respect to $x$ and $\|\cdot\|$ is any  norm in finite dimension. Throughout the paper, the variable $x$ plays no role, and we often write $f(\cdot, \theta )$. The $p$-th derivative $f^{(p)}(x, \theta)$ is always taken with respect to the variable $\theta $.

We write $F_n$ for the empirical distribution, that is, if $X_1, \dots, X_n$ are independent with distribution $F(\cdot, G)$, then $F_n(t) = \frac1n\sum_{i=1}^n \1 _{ \{X_i\le t\}}$.

\medskip

As usual, the ($L^1$)-transportation distance, or Wasserstein metric, is used to compare two mixing distributions $G_1$ and $G_2$. It completely bypasses identifiability issues that would arise with the square error on parameters.
The definition is:
\begin{equation}
\label{defW}
  W(G_1,G_2)=\inf_{\Pi}\int_{\Theta\times\Theta}|\theta-\theta'|\dd \Pi(\theta,\theta'),
\end{equation}
where the infimum is taken over probability measures $\Pi$ on $\Theta\times\Theta$ with marginals $G_1$ and $G_2$. By the Kantorovich-Rubinstein dual representation
\citep[e.g.][section 11.8]{Dudley}, $ W(G_1,G_2)$ can be viewed as  a supremum:
\begin{equation}
  \label{dualW}
   W(G_1,G_2)=\sup_{|f|_{\mathrm{Lip}}\le 1}\int_{\Theta}f(\theta)\dd (G_1-G_2)(\theta),
\end{equation}
where $|f|_{\mathrm{Lip}}$ stands for the Lipschitz seminorm of
$f$. Endowed with the metric $W$, the space $\Glm$ is compact. It is sometimes convenient to use the notation $W(G_1-G_2)$
instead of $W(G_1,G_2)$.


We also introduce the  Wasserstein $\eps$-ball of a mixing distribution $G_0$: 
\[\mathcal{W}_{G_0}(\eps)=\{G\in\Gf : W(G, G_0)
< \eps\}.\]

\medskip

In the rest of the paper, we will need to compare sequences, say $(a_n)$ and $(b_n)$.
The notation $a_n\lec b_n$ (or even $a\lec b$ if $n$ is kept
implicit) means that there is a  positive constant
$C$ such that $a_n\le C b_n$; in other words,
$a_n=O(b_n)$. We will also use  $a_n\asymp b_n$  for $b_n \lec a_n\lec
b_n$.  If we need to stress the dependence of the constants $C$
on other parameters, say $C = C(u,v,\theta)$, we will write 
$a_n \underset{u,v,\theta}{\lec} b_n$ or $a_n \underset{u,v,\theta}{\asymp} b_n$. 

Below $\xrightarrow[]{d}$ (resp. $\xrightarrow[]{P}$) stands for convergence in distribution (resp. in probability). We write $\lb i,j \rb$ for the set of integers between $i$ and $j$.

\subsection{Regularity: $(p,q)$-smoothness} 
It is notationally natural to set 
\begin{equation}
  \label{Ff}
 F(x,\theta )=\int_{-\infty}^x f(y,\theta )\dd\lambda(y), 
\end{equation}
and to denote by $\mathbb{E}_{\theta}$ the expectation
w.r.t. $f(x,\theta)\dd\lambda(x)$. If we identify $\theta$ with the
Dirac measure $\delta_\theta$, the notations extend naturally to mixing
distributions $G$ by linearity. 

Recall that 
 derivatives $f^{(p)}$ are taken
w.r.t. the variable $\theta$. 

\begin{defin}
    \label{Eia}
Set for $p\in \N$ and $q>0$,
    \begin{align}
        \label{Epq}
        E_{p,q}\left(\theta, \theta', \theta''\right) & = \mathbb{E}_{\theta}\left|\frac{ f^{(p)}(\cdot, \theta')}{f(\cdot,\theta'')} \right|^q. 
    \end{align}
    We say that $\left\{ f(\cdot,\theta),\theta \in \Theta \right\}$ is \emph{$(p,q )$-smooth} if
    \begin{enumerate}
    \item $E_{p,q}$ is a well-defined $[0,\infty]$-valued continuous function on $\Theta^3$,
  \item \label{proche} There exists
    $\varepsilon > 0$ such that 
\begin{align*}
        |\theta' - \theta''| < \varepsilon & \implies \forall \theta\in\Theta,\quad E_{p,q }(\theta, \theta', \theta'') < \infty.
    \end{align*}
    \end{enumerate}   
    \end{defin}

    These smoothness conditions are easy to check in practice, and general enough. For example, all exponential families satisfy them, as shown in Section~\ref{exp_smooth} in the supplemental part \citep{Supp}.

    They will be useful for proving local asymptotic normality \citep{LeCam} of relevant families.

\subsection{Regularity: $k$-strong identifiability}

\citet{Chen} introduced a notion of strong identifiability. We will need a slightly more general version.
\begin{defin}
    \label{identifiability}
The family $\left\{ F(\cdot,\theta ), \theta \in \Theta  \right\} $ of
    distribution functions is \emph{$k$-strongly identifiable} if for any finite set of say $m$ distinct $\theta_j\in\Theta$, then the equality
    \begin{align*}
        \norm{ \sum_{p=0}^k \sum_{j=1}^m \alpha _{p,j} F^{(p)}(\cdot, \theta _j) }_\infty = 0
    \end{align*}
    implies $\alpha _{p,j} = 0$ for all $p$ and $j$.
\end{defin}

\citeauthor{Chen}'s strong identifiability corresponds to $2$-strong identifiability. Let us exemplify why this notion is useful. Consider a sequence of mixing densities $G_n = \frac12 (\delta_{n^{-1}} + \delta_{-n^{-1}})$. Then, if we can develop around $\theta = 0$, we see that $F(\cdot, G_n) = F(\cdot, 0) + n^{-2} F^{(2)}(\cdot, 0) + o(n^{-2})  $. Then $2$-strong identifiability ensures that $\left\lVert F(\cdot, 0) - F(\cdot, G_n) \right\rVert _{\infty} $ is of order $n^{-2}$, as shown in Proposition~\ref{infimum} below, whereas simple identifiability would say nothing. We will need $k$-strong identifiability when more moments in $\theta $ w.r.t. the two mixing distributions are the same.

\begin{prop}
    \label{infimum}
Fix $m\ge 1$. Let $\left\{ F(\cdot,\theta ), \theta \in \Theta  \right\} $ be
$k$-strongly identifiable. For $\eps>0$, the $\theta _i$ are $\varepsilon $-separated if they belong to
\[\mathcal{D}_\eps=\left\{(\theta _i)_{1\le i\le j}: \forall i\ne
  i',\quad |\theta_i-\theta_{i'}|\ge \eps\right\}.\]
If $F^{(m)}(x,\theta)$ is continuous in $\theta$, then
\begin{align}
        \label{alpha_bound_sep}
\forall   (\theta _i)_{1\le i\le j}\in  \mathcal{D}_\eps,\quad   \norm{ \sum_{p=0}^k \sum_{j=1}^m \alpha _{p,j} F^{(p)}(\cdot, \theta _j) }_\infty \underset{\eps}{\gec} \norm{\alpha }.
    \end{align}   
\end{prop}
\begin{proof}
       Set $\alpha=(\alpha _{p,j})$
 and $\vartheta=(\theta _i)$. The $[0,\infty]$- valued function $(\alpha,\vartheta)\mapsto
 \norm{ \sum_{p=0}^k \sum_{j=1}^m \alpha _{p,j} F^{(p)}(\cdot, \theta _j)
 }_\infty$ is lower semi-continuous on the compact set
 $\{\alpha :\|\alpha\|=1\}\times  \mathcal{D}_\eps$ so that it admits a minimum. By $k$-strong identifiability, it is nonzero.
\end{proof}
We expect the strong identifiability to be rather generic, and hence
the statements of this paper often meaningful. In particular, \citet[Theorem~3]{Chen} has proved that 
location and scale families with smooth densities are $2$-strongly identifiable. The theorem and the proof straightforwardly generalise to our case. We merely state the result.
\begin{thm}
    \label{thm_identifiability}
    Let $k\ge 1$. Let $f$ be a probability density with respect to to the Lebesgue measure on $\mathbb{R} $. Assume that $f$ is $k-1$ times differentiable with 
\[\lim_{x\to\pm \infty}f^{(p)}(x)=0\text{ for } p\in \lb 0,k-1\rb.\]
Consider $f(x,\theta)= f(x-\theta)$, with $\theta\in\Theta\subset\R$. Then the corresponding distributions family $\{F(\cdot,\theta),\theta\in\Theta\}$ is $k$-strongly identifiable.  If $\Theta\subset (0,\infty)$, the result stays true with $f(x,\theta)=\frac{1}{\theta}f\left(\frac{x}{\theta}\right)$.
\end{thm}

See also the article by \citet{Hol} for more general conditions, that also generalize well to $k$-strong identifiability.

\subsection{Assumptions}

For proving lower bounds on rates, we will assume:
\begin{assumA}
\label{as:smooth}
  The  family of densities  $\{f(\cdot,\theta),\theta\in \Theta\}$
    satisfies, with $G_0\in \Gm0$,
    \begin{itemize}
    \item $(p,q)$-smoothness for all $(p,q)\in\lb 1,2(m-m_0)+2\rb\times\lb
    1,4\rb$,
  \item There is some support point $\theta_0$ of $G_0$ such that
     $\theta_0\in \interior{\Theta}$ and
\[ \int|f^{(2(m-m_0)+1)}(\cdot,\theta_0)|\dd\lambda >0.\] 
    \end{itemize}   
\end{assumA}
These conditions allow to prove local asymptotic normality \citep{LeCam} for relevant families. This will give some insight on the reason why the lower bound on the rate holds, and on how the mixtures behave when we change the parameters in the least sensitive direction. The condition on the support point guarantees identifiability locally for the families, and we need more derivatives than usual, since there will be cancellations in the first terms.

For proving upper bounds on rates we will assume:
\begin{assumB}
 The family of densities $\{f(\cdot,\theta),\theta\in\Theta\}$ satisfies, with $F(x,\theta)=\int_{-\infty}^xf(\cdot,\theta)\dd\lambda$,
\begin{itemize}
\item For all $x$, $F(x,\theta)$ is $k$-differentiable w.r.t. $\theta$,
\item $\{F(\cdot,\theta),\theta\in\Theta\}$ is $k$-strongly identifiable,
\item There is a uniform continuity modulus $\omega(\cdot)$ such that 
\[\sup_x\big|F^{(k)}(x,\theta_2)-F^{(k)}(x,\theta_1)\big|\le \omega(\theta_2-\theta_1)\]
with $\lim_{h\to 0}\omega(h)=0$.
\end{itemize}
\end{assumB}
Notice that the latter condition is  satisfied  if $F^{(k+1)}$ exists and is bounded.

These derivability conditions should be compared with the usual parametric case, where differentiability in quadratic mean, or twice differentiability in $\theta $ for a less technical condition, is enough to get $n^{-1/2}$ local minimax rate. We will need $B(2m)$ to prove a global minimax rate of $n^{-1/(4m - 2)}$, and $B(1)$ for a pointwise rate of $n^{-1/2}$ everywhere.

\section{Main results}
\label{sec:main}

We now have the tools to state the main results.

Keeping in mind the following viewpoint will help getting intuition on the results. The data we have access to is the empirical distribution $F_n$, which gets closer to the true mixture $F(\cdot, G)$ at rate $n^{-1/2}$. Hence $G_1$ and $G_2$ can be told apart if $\left\lVert F(\cdot, G_1) - F(\cdot, G_2) \right\rVert_{\infty}$ is at least of order $n^{-1/2}$.

If we get a control on powers of the transportation distance $W(G_1, G_2)^d$ by $\left\lVert F(\cdot, G_1) - F(\cdot, G_2) \right\rVert_{\infty}$, we then get $n^{-1/(2d)}$ rates.  For upper bounds, Lemma~\ref{lemDKW} makes this rigorous.

For lower bounds, general estimators could hope to do better, say by noticing that some data points are not in the support of some $G_1$. However this will not be the case under sufficient smoothness conditions.

In this setting, the minimum distance estimator discussed by \citet{Deely&Kruse} and \citeauthor{Chen} is natural, and we often use it later on. 
\begin{defin}
    \label{min_dist}
    The \emph{minimum distance estimator} 
 $\widehat{G}_n \in \Glm $ is any mixing distribution whose corresponding mixture minimizes the $L^{\infty}$-distance to the empirical distribution, that is:
\begin{equation}
  \label{gn}
\|F(\cdot,\widehat{G}_n)-F_n\|_\infty= \inf_{G\in \Glm}\|F(\cdot,G)-F_n\|_\infty.  
\end{equation}
Note that the infimum is attained since $G\mapsto
\|F(\cdot,G)-F_n\|_\infty$ is lower semi-continuous on the compact
metric space $(\Glm,W)$.
\end{defin}

\subsection{Local asymptotic minimax rate}
\label{sub:minimax}

When the number $m$ of components in a mixture is exactly known and $f(\cdot, \theta )$ is smooth in $\theta $, we are in a simple smooth parametric case, with $2m-1$ parameters. Hence the optimal local minimax rate of estimation is $n^{-1/2}$ in mean square error, with a constant given by the Cramér-Rao bound \citep{hajek1972local}. This translates to the same rate in transportation distance.

In particular the minimum distance estimator introduced above attains the $n^{-1/2}$ rate (Theorems~\ref{equal_compo} and~\ref{choose_and_estimate}), not necessarily with the optimal constant.

The difficulty with mixtures stems from what happens when the number of components is not known: is there only one component here, or two very close ones? If there are two, what are their weights and how far apart are they? 

We can build families of mixtures that are very hard to tell apart, because their mixing distributions have the same first moments. Indeed, suppose that all the support points of the mixture are of the form $\theta _0 + h_j$ with $h_j$ small. Then a Taylor expansion in $\theta $ of the mixture $F(\cdot, G)$ yields:
\begin{align*}
    F(\cdot, G) & = \sum_{p=0}^k \sum_j (\pi_j h_j^p) \frac{F^{(p)}(\cdot, \theta_0)}{p!} + o(\pi_j h_j^{k}). \\
\end{align*}
So that, according to our heuristics on the empirical distribution, if $G_1$ and $G_2$ have the same $k$ first moments, we cannot tell them apart if $h_j^k \ll n^{-1/2}$, that is if $h_j \ll n^{-1/(2k)}$.

As an example, let us consider two-component mixtures around $G_0 = \delta _0$.  Then $G_{1,n} = \frac{1}{2} \left( \delta
      _{- 2 n^{-1/6}} + \delta _{2n^{-1/6}} \right) $ and $G_{2,n} =
    \frac{4}{5} \delta _{-n^{-1/6}} + \frac{1}{5} \delta _{4
      n^{-1/6}}$ both have $0$ as first moment, and $4n^{-1/3}$ as
      second moment. The third moments are respectively zero for $G_{1,n}$ and $12n^{-1/2}$ for $G_{2,n}$. According to this heuristics, no test can reliably tell $G_{1,n}$ from $G_{2,n}$ with an $n$-sample. On the other hand, we clearly have $W(G_{1,n} ,G_{2,n})= n^{-1/6} $ for all $n$. So that the minimax  rate for $2$-mixtures cannot be better than $n^{-1/6}$.

This moment matching argument can be made rigorous and precise with two tools. One is  \citeauthor{Lindsay}'s \citeyearpar{Lindsay} Hankel trick (Theorem~\ref{Lindsay}), also used  by \cite{Gass} to estimate the order of a mixture. The other is local asymptotic normality (Definition~\ref{defLAN}), developed by \citet{LeCam}.  We use them to build a one-parameter locally asymptotically normal family with scale factor $n^{1/ (4(m - m_0) + 2)}$ in Theorem~\ref{LAN}, which will entail:
\begin{thm}
    \label{lower_bound}
Let $G_0\in\Gm0$ and set $\eps_n=n^{-1/ (4(m - m_0) +
  2)+\kappa}$ for any $\kappa >0$. Under Assumption~A, for any
sequence of estimators $\widehat{G}_n$ based on i.i.d. $n$-samples,
    \begin{align*}
        \liminf_{n\to \infty} \sup_{G_1\in \Gm{} \cap\mathcal{W}_{G_0}(\eps_n)}n^{1/ (4(m - m_0) + 2)} \,\mathbb{E}_{G_1}\left[W(G_1, \widehat{G}_n)\right]& > 0.
    \end{align*}
\end{thm}

  
\medskip

Theorem~\ref{lower_bound} gives a lower bound on the local asymptotic minimax rate of estimation. The corresponding upper bounds, both local and global, are given by the following theorem: 
\begin{thm}
    \label{main}
    Let $G_0\in\Gm0$. Then, under Assumption~B(2m), there is an $\eps >0$ such that the minimum distance estimator \eqref{gn} in $\Glm$ satisfies
    \begin{align}
        \label{main1}
    \mathbb{E}_{G_1}\left[ W(\widehat{G}_n,G_1) \right] &\lec \frac{1}{n^{1/(4(m-m_0)+2)}}
    \end{align}
 uniformly for $G_1$ in $\Glm\cap \mathcal{W}_{G_0}(\eps)$, where $n$ is the sample size.

 Moreover, uniformly for $G_1$ in $\Glm$,
 \begin{align}
     \label{main2}
     \mathbb{E}_{G_1}\left[ W(\widehat{G}_n,G_1) \right] & \lec \frac{1}{n^{1/(4m - 2)}}.
\end{align}
\end{thm}

We prove it by establishing a uniform control of $W(G_1, G_2)^{2m - 2m_0 + 1}$ by $\left\lVert F(\cdot, G_1) - F(\cdot, G_2) \right\rVert_{\infty}$ in Theorem~\ref{orders}.

Obtaining this control is quite technical, however.  To do so, we consider sequences of couples $(G_{1,n}, G_{2,n})$ minimizing the relevant ratios, and express $F(\cdot, G_{1,n}) - F(\cdot,G_{2,n})$ as a sum on their components $F(\cdot, \theta_{j,n})$ and relevant derivatives. A difficulty arises: distinct components $\theta_{j,n}$ may converge to the same $\theta_j$, leading to cancellations in the sums. Forgetting this case was the mistake by \citet{Chen} in the proof of their Lemma 2. We overcome the issue in Section~\ref{tree} by using a coarse-graining tree: each node corresponds to sets of components whose pairwise distance decrease at a given rate. We may then use Taylor expansions on each node and its descendants, while ensuring that we keep non-zero terms (Lemma \ref{lemrec}). 

\begin{rems}
    \label{rem_minimax}
    \begin{itemize}
        \item{Theorems~\ref{lower_bound} and~\ref{main} together imply that the optimal local asymptotic minimax rate is $n^{-1/(4(m-m_0)+2)}$ for estimating a mixture with at most $m$ components around a mixture with $m_0$ components.}
    \item{The rate is driven by $m - m_0$, that is, it gets harder to estimate the parameters of a mixture when it is close to a mixture with less components.}
\item{The worst case is when $m_0 = 1$, yielding a global minimax rate of estimation $n^{-1 / (4 m + 2)}$. The rate gets worse when more components are allowed. So that the nonparametric rates for estimating mixtures with an infinite number of components like in deconvolution appear natural.}
\item{On the other hand, when the number of components is known, that is $m = m_0$, we have the usual local minimax rate $n^{-1/2}$.}
\item{The global minimax rate on the mixtures with exactly $m$ components $\mathcal{G} _m$ stays at $n^{-1 / (4 m + 2)}$, because $\mathcal{G} _m$ is not compact, and Theorem~\ref{lower_bound} still apply in the vicinity of $m_0$-component mixtures.}
\end{itemize}    
\end{rems}

The slower rate $ n^{-1 / (4 m + 2)}$ might be a little surprising when for example some Bayesian estimators have $n^{-1/4}$ rate of convergence \citep{Ish}. However this convergence rate is not the local minimax rate, but is closer to a pointwise rate of convergence, that is the speed at which an estimator converges to a fixed $G$ when $n$ increases. The difference with local minimax may be viewed as the loss of uniformity in $G$. We now study the optimal pointwise rates everywhere.

\subsection{Pointwise rate and superefficiency}
\label{sub:pointwise}

One motivation for local minimax results was to make clear how the Hodges' estimator \citep[ch.8]{VDV} and other superefficient estimators could cohabit with Cramér-Rao bound, and how much they could improve on it.

Specifically, a superefficient estimator can have a better pointwise convergence rate than any regular estimator, but not a better local minimax convergence rate \citep{hajek1972local}. Moreover, it turns out that they can only have a better pointwise rate on a Lebesgue-null set \citep[ch.8]{VDV}.

Now, the set of parameters of mixtures with less than $m$ components
$\mathcal{G} _{< m}$ is a Lebesgue-null set among those of mixing distributions 
with at most $m$ components $\mathcal{G} _{\le m}$. Hence, we might expect that, by biasing the estimators toward the low numbers of components, we might attain better pointwise rates on $\mathcal{G} _{< m}$, up to $n^{-1/2}$, which is the value when the number of components is known. By letting $m$ go to infinity, we would have this pointwise rate for all finite mixing distributions. It turns out this is indeed the case.

\medskip

An estimator achieving $n^{-1/2}$ rate may be built from minimum distance estimators \eqref{gn}.
For all $m$ we denote by $\widehat{G}_{n,m}$ the minimum distance estimator in $\Glm$.
For any fixed $\kappa\in (0,1/2)$, set
  \begin{equation}
 \label{est_fin}   
\widehat{G}_n  = \widehat{G}_{n,\hat{m}},
  \end{equation}
with
\begin{equation}
 \label{nb_comp} 
 \hat{m} = \hat{m}_n=\inf \left\{ m\ge 1 :  \| F(\cdot, \widehat{G}_{n,m}) - F_n \|_{\infty} \le n^{-1/2 + \kappa } \right\}.
\end{equation}
Since the typical distance between empirical and cumulative distribution functions is $n^{-1/2}$, this $\widehat{m} $ is the lowest number of components that is not clearly insufficient. 

We will obtain:
\begin{thm}
    \label{choose_and_estimate}
    Under Assumption~B(1), for any finite mixing distribution $G_0 \in \mathcal{G} _{< \infty}$,
   \begin{align*}
       \mathbb{E}_{G_0} \left[ W(\widehat{G}_n, G_0) \right] & \lec n^{-1/2}.  
   \end{align*}
   Notice that the above inequality is not uniform in $G_0$.
\end{thm}

\begin{rems}
    \label{plus}
    \begin{itemize}
        \item{The rate $n^{-1/2}$ cannot be improved since it is the rate if the number of components is known beforehand.}
        \item{This is slightly stronger than just checking that we find the right number of components and then applying Theorem~\ref{main}, because we need much less regularity. Only Assumption~B(1) is required, instead of B(2m). That is, we do not need more smoothness when the number of components increases. Under the hood we rely on the bound in Theorem~\ref{equal_compo} instead of Theorem~\ref{orders}.}
        \item{The estimation of the number of components $\hat{m}$  and the estimation of $\widehat{G}$ within $\mathcal{G} _{\hat{m}}$ are not associated. For example, we may estimate $\hat{m}$ with Equation \eqref{nb_comp}, and then use the maximum likelihood estimator $\widehat{G}$ on $\mathcal{G} _{\hat{m}}$. Conversely, we may estimate the number of components using \citeauthor{MR3015722}'s \citeyearpar{MR3015722} penalized maximum likelihood estimator.}
\end{itemize}
\end{rems}

\subsection{Interpretation and practical consequences}
\label{practical}

Disagreement between local minimax and pointwise rates everywhere might be rare enough that it is worth recalling what it means.

At a given point $G$, the asymptotic rate of convergence to $G$ will be the pointwise rate $C(G) n^{-1/2}$. However, the estimator will enter this asymptotic regime only after a long time. More precisely, it enters this regime after that $G$ is not anymore in any of the balls used in the local minimax bound.
Alternatively, we may view this situation as the constant $C(G)$ exploding when $G$ is close to certain $G_0$.

In our case, imagine we have a mixture with three components, all within distance $\delta $ of $\theta _0$. Then about $\delta ^{-(4(3 -1) + 2)} = \delta ^{-10}$ data points are necessary to get an estimator with an error of $\delta $.

In particular, if $G_1$ and $G_2$ are two such three-component mixtures, chosen to have the same first four moments, and $\tilde{G}_1$ and $\tilde{G}_2$ are the same mixtures, rescaled to be ten times closer, we will need $10^{10}$ as many data points to tell them apart as for $G_1$ and $G_2$.

As a consequence, if the components of the mixture to be estimated are not far apart one from the other, it is quite often impossible to get enough data points to get an appropriate estimate.

An experimentalist with any leeway in what he measures (use of different markers, say) might then wish to ensure that the peaks are far apart, even at the cost of many data points.

\subsection{Further work}
\label{extensions}

This article contains the proof that the optimal local minimax rate of estimation around a mixture with $m_0$ components among mixtures with $m$ components is $n^{-1/(4(m -m_0) + 2)}$, when the parameter space $\Theta $ is a compact subset of $\mathbb{R} $.

We think that extension to a multivariate $\Theta$ should be easy enough, much like \citet{Ngu} did for the former erroneous result. On the other hand, non-compactness of $\Theta $ would probably bring about technical difficulties, and cases where the result would not hold. Stronger forms of identifiability would probably be required in general, to avoid problems with limits.

Finally, another line of inquiry are the results that might be expected in a Bayesian framework. The most natural equivalent to the convergence rate of the \emph{a posteriori} distribution to the real parameter is the pointwise rate of convergence. Hence the question: can we build Bayesian estimators where the \emph{a posteriori} distributions converge at rate $n^{-1/2}$ everywhere? Of course, the convergence would not be uniform.

\section{Key tools}
\label{key}

\subsection{Local asymptotic normality and Theorem~\ref{lower_bound}}
\label{sec:lowerbound}

We prove Theorem~\ref{lower_bound} by displaying local asymptotic families $\{G_n(u), u\in \mathbb{R} \}$ with scale factor $n^{-1/(4(m - m_0) +2)}$.
 A far-from-general definition, but sufficient for our purposes, of local asymptotic normality \citep{LeCam} is as follows:
\begin{defin}
    \label{defLAN} Given densities $f_{n,u}$ ($n\in\N,u\in \R$) with
    respect to some dominating  measure, consider experiments
    $\mathcal{E} _n =\left\{ f_{n,u}, u\in \mathcal{U} _n \right\} $
    where the $\mathcal{U} _n,n\in\N,$ are real sets such that each real number be in $\mathcal{U}_n$ for $n$ large enough. Let $X$ have density $f_{n,0}$ and consider the log-likelihood ratios:
    \begin{align*}
        Z_{n, 0}(u) & = \ln \left( \frac{f_{n, u}(X)}{f_{n,0}(X)}  \right) .
\end{align*}
    Suppose that there is a positive constant $\Gamma $ and a sequence of random variables $Z_n$ with $Z_n \xrightarrow[]{d} \mathcal{N} (0, \Gamma )$, such that for all $u\in \mathbb{R} $,
    \begin{align}
    \label{ELAN}
    Z_{n, 0} (u) - u Z_n + \frac{u^2}{2} \Gamma & \xrightarrow[n\to\infty]{P}  0.
\end{align}
The sequence of experiments $\mathcal{E} _n $ is said \emph{locally asymptotically
  normal} (LAN) and \emph{converging} to the Gaussian shift experiment $\left\{ \mathcal{N} (u\Gamma, \Gamma ), u \in \mathbb{R}  \right\} $.
\end{defin}

Note that if $X$ were exactly $\mathcal{N} (u\Gamma, \Gamma
)$-distributed,  the l.h.s of \eqref{ELAN} would be zero, with a
suitable $Z_n$ exactly $\mathcal{N} (0, \Gamma )$-distributed. In
addition, intuitively, (almost) anything that can be done in a
Gaussian shift experiment can be done asymptotically in a  LAN sequence of experiments.

\medskip

Now, consider a mixing distribution $G_0  = \sum_{j=1}^{m_0}  \pi_j \delta_{\theta _j} \in \mathcal{G} _{m_0}$ with $m_0$-th support  point $\theta _{m_0}$ 
   in the interior of $\Theta$.

   Then, for $n$ big enough, $\theta _{j,n}(u) = \theta _{m_0} +
n^{-1/(4d-2)} h_j(u) $ is in $\Theta $. The LAN family will be 
\begin{equation}
  \label{mixdist}
 G_n(u)  = \sum_{j=1}^{m_0 - 1} \pi_j \delta_{\theta _j} + \pi_{m_0}
\sum_{j= m_0}^m \pi_j(u) \delta _{\theta_{j,n}(u)},
\end{equation}
with $h_j(u)$ and $\pi_j(u)$ chosen so that the first moments around $\theta _{m_0}$ are the same for all $u$, that is, for some relevant $\mu_k$,
\begin{align*}
    \sum_{j=m_0}^m \pi_j(u) h_j(u)^k & = \mu_k & \text{for $k \leq 2(m - m_0)$,} \\
    \sum_{j=m_0}^m \pi_j(u) h_j^{2(m - m_0) + 1} &= u,
\end{align*}
The reason why they exist is Theorem~2A by \citet{Lindsay} on the matrix of moments:
\begin{thm}
    \label{Lindsay}
Given numbers $1,\mu_1,\ldots,\mu_{2d}$, write $M_k$ for the
 $k+1$ by $k+1$ (Hankel)  matrix with entries $(M_k)_{i,j} =
 \mu_{i+j-2}$ for $k\in\lb 1,d\rb$. 
    \begin{enumerate}[a.]
    \item The numbers $1, \mu_1, \dots, \mu_{2d}$ are the moments of a
    distribution with exactly $p$ points of support if and only if
    $\det M_k > 0$ for $k\in\lb 1,d-1\rb$ and $\det M_p = 0$.
  \item\label{lind} If the numbers $1, \mu_1, \dots, \mu_{2d-2}$ satisfies $\det M_k > 0$ for $k\in\lb 1,d-1\rb$ and $\mu_{2d-1}$ is any scalar, then there exists a unique distribution with exactly $d$ points of support and those initial $2d-1$ moments. 
    \end{enumerate} 
\end{thm}
   
With such a family, we can prove the following theorem, whose proof we delay to Section~\ref{proofs}:
    \begin{thm}
        \label{LAN}
       Let $G_0  = \sum_{j=1}^{m_0}  \pi_j \delta_{\theta _j} \in \mathcal{G} _{m_0}$ be a mixing distribution whose $m_0$-th support point is
   in the interior of $\Theta$. Let $m \ge  m_0$. 
  
    Then there are mixing distributions $G_n(u)$ ($n\ge 0,u\in\R$)
    all in
    $\Gm{}$ such that
    \begin{enumerate}[a.]
        \item \label{lan1} $W(G_{n}(u), G_0) \to 0$ for all $u\in\R$. More precisely,
\[W(G_{n}(u), G_0) \underset{u}{\lec} n^{-1/ (4(m -m_0) +2)};\]
    \item \label{lan2} The mixing distributions $G_n(u)$ get closer at rate $n^{-1/(4(m
        - m_0) + 2)} $: for all $u$ and $u'$, 
      \[W(G_{n}(u), G_{n}(u')) \underset{u,u'}{\gec} n^{-1/ (4(m -m_0) +
        2)};\]
    \item \label{lan3} If the family $\left\{ f(\cdot, \theta ), \theta \in
          \Theta  \right\} $ satisfies Assumption~A with
        $\theta_0=\theta_{m_0}$, then there is a number $\Gamma >0$, a sequence $U(n) \to\infty$ and an infinite subset $\N_0$ of $\mathbb{N} $ along
      which the experiments 
\[\mathcal{E} _n = \left\{
        \bigotimes_{i=1}^nf\left(\cdot,G_n(u)\right), |u|
        \le U(n)\right\} \]  converge to the Gaussian shift
      experiment $\left\{ \mathcal{N} (u\Gamma, \Gamma), u \in \mathbb{R}
      \right\}$.
     \end{enumerate}
    \end{thm}
   
\begin{rem}
    \label{remLAN}
 We want only an example of this slow convergence, and it should be somewhat typical. That is why we have chosen the regularity conditions to make the proof easy, while still being easy to check, in particular for exponential families. 
    
    In particular, it could probably be possible to lower $q$ in
    $(p,q)$-smoothness to $2+\eps$ and still get the uniform bound we
    use in the law of large numbers below. Similarly, less derivability might be necessary if we tried to imitate differentiability in quadratic mean.
    
    In the opposite direction the variance $\Gamma $ in the limit
    experiment is really expected to be $\pi_{m_0}^2\mathbb{E}_{G_0}
    \left| \frac{f^{(2d-1)}(\cdot, \theta _{m_0}) }{f(\cdot, G_0)} \right|^2$ in most cases, but more stringent regularity conditions may be needed to prove it.
\end{rem}

Theorem \ref{LAN} and its proof show that when the first moments of the components of
    the mixing distribution $G$ near $\theta _{m_0}$ are known, all
    remaining knowledge we may acquire is on the next moment, and
    that's the ``right'' parameter: it is exactly as hard to make a
    difference between, say, $10$ and $11$ as between $0$ and $1$.

    On the other hand, for our original problem the cost function is the transportation distance between mixing distributions. So that an optimal estimator in mean square error for $u$ is not optimal for our original problem. Moreover just taking the loss function $  c(u_1,u_2)$ in the limit experiment runs into technical problems since this might go to zero as $u_2$ goes to infinity. They could be overcome, but it is easier to show how Theorem~\ref{LAN} entails Theorem~\ref{lower_bound} using just two points and contiguity \citep{LeCam1}.

\begin{proof}[Proof of  Theorem~\ref{lower_bound}]
Fix $u > 0$ and consider the densities $f_{n,u}=\otimes_{i=1}^nf\left(\cdot,G_n(u)\right)$  with associated probability
    measures $\P_{G_n(u)^{\otimes n}}$, which are simply denoted by $\P_{G_n(u)}$ below. We have 
    \begin{equation}
      \label{eq:liminf}
      \liminf_{n\to \infty}\inf_{A:\P_{G_n(0)}(A)\ge 3/4}\P_{G_n(u)}(A)\ge
\frac14\e^{-\frac{u^2}{2}\Gamma}.
    \end{equation}
Indeed, from Theorem~\ref{LAN}.\ref{lan3}. and the LAN property \eqref{ELAN}, if  $X$ is of density $f_{n,0}$ and $n$ ranging over $\N_0$, then   
\[\rho_n=\frac{f_{n,u}(X)}{f_{n,0}(X)}\e^{-uZ_n+\frac{u^2}{2}\Gamma}\xrightarrow[]{P}1\quad\text{where}\quad Z_n\xrightarrow{d}{}\mathcal{N}(0,\Gamma).\]
For any event $A$, 
\begin{equation*}
  \P_{G_n(u)} (A)  = \E_{G_n(0)} \left(\frac{f_{n,u}(X)}{f_{n,0}(X)}\1_A\right) =\e^{-\frac{u^2}{2}\Gamma}\, \E_{G_n(0)} \left(\rho_n\e^{uZ_n}\1_A\right).
\end{equation*}
Furthermore, by restriction on the event $\{Z_n>0\}$ and by using
$\rho_n \xrightarrow[]{P}1$, we get 
\begin{equation*}
 \E_{G_n(0)} \left(\rho_n\e^{uZ_n}\1_A\right)\ge\P_{G_n(0)} (A)-\P_{G_n(0)} (Z_n\le 0)+o(n).
\end{equation*}
Taking now the infimum on events $A$ such that $\P_{G_n(0)} (A)\ge 3/4$ and passing to the limit as $n\to \infty$ along $\N_0$, we obtain \eqref{eq:liminf}.

We now consider, for any sequence of estimators $\widehat{G}_n$, the event 
\[A =  \{ n^{1/ (4(m - m_0) + 2)} W(G_{n}(0), \widehat{G}_n) \ge  a\}\] 
for some $a>0$ to choose.  By Theorem~\ref{LAN}.\ref{lan2}., there is a
constant $c(u,0)>0$ such that $n^{1/ (4(m - m_0) + 2)}
W(G_n(u),G_{n}(0)) \ge c(u,0)$ so that  by the triangle's inequality,
\[A^c\subset  \{n^{1/ (4(m - m_0) + 2)} W(G_{n}(u), \widehat{G}_n) \ge 
c(u,0) -a\}.\]
Choose $a=c(u,0)/2$. Then either  $\P_{G_n(0)} (A) \ge  1/4  $, which gives 
\[\sup_{G_1 \in \{G_n(0)\}} n^{1/ (4(m
  - m_0) + 2)}\mathbb{E}_{G_1} W(G_1, \widehat{G}_n)
\ge \frac{a}{4},\]
 or, by \eqref{eq:liminf}, $\P_{G_n(u)} (A^c) \ge    \e^{-\frac{u^2}{2}\Gamma}/4$ in
the limit  so that 
\[ \liminf_{n\to \infty} \sup_{G_1 \in \{G_n(u)\}} n^{1/ (4(m
  - m_0) + 2)}\mathbb{E}_{G_1} W(G_1, \widehat{G}_n)
\ge \frac{a}{4}\e^{-\frac{u^2}{2}\Gamma}.\]
Thus,  gathering the two inequalities, we get
\[ \liminf_{n\to \infty} \sup_{G_1 \in \{G_n(0), G_n(u)\}} n^{1/ (4(m
  - m_0) + 2)}\mathbb{E}_{G_1} W(G_1, \widehat{G}_n)
\ge \frac{a}{4}\e^{-\frac{u^2}{2}\Gamma}.\]
Note to finish that by Theorem~\ref{LAN}.\ref{lan2}., each $G_n(0)$ or
$G_n(u)$ is at Wasserstein distance at most  $n^{-1/ (4(m - m_0) + 2)+\eps}$
from $G_0$, for large $n$ enough. Theorem~\ref{lower_bound} is thus established.
\end{proof}

\subsection{Comparison between distances and upper bounds on convergence rates}
\label{sec:upperbound}

All our upper bounds on convergence rates come from the properties of the minimum distance estimator~\eqref{gn}. Now, by the triangle's inequality, if the $n$-sample comes from $F(\cdot, G_1)$ with $G_1 \in \Glm$, then 
\begin{align}
      \label{fois2}
\|F(\cdot,\widehat{G}_n)-F(\cdot,G_1)\|_\infty\le 2\|F(\cdot,G_1)-F_n\|_\infty.
  \end{align}
  Hence, the following lemma allows us to get bounds on rates whenever  we can control (a power of) the transportation distance between mixing distributions by the $L^{\infty}$-distance between mixtures.
\begin{lem}
\label{lemDKW}
  Let $d\ge 1$. Assume that  the optimal estimators $\widehat{G}_n$ in Theorem~\ref{main} satisfy for some constant $C>0$ and on some  event $A$,
\[W(\widehat{G}_n,G_1)^d\le C\|F_n-F(\cdot, G_1)\|_\infty.\]
Then 
\[\E_{G_1}W(\widehat{G}_n,G_1)\le \left(\frac{\pi C^2}{2}\right)^{1/2d}n^{-1/2d}+\Diam(\Theta)\P_{G_1}(A^c).\]
\end{lem}
\begin{proof}
  By assumption, we can bound $W(\widehat{G}_n,G_1)$ on $A$ and we can also always bound $W(\widehat{G}_n,G_1)$ by $\Diam(\Theta)$, in particular on $A^c$, so that using Jensen inequality,
  \[\E_{G_1}W(\widehat{G}_n,G_1)\le  C^{1/d}\left[\E_{G_1}\|F_n-F(\cdot, G_1)\|_\infty\right]^{1/d} +\Diam(\Theta)\P_{G_1}(A^c). \]
Now, the Dvoretzky-Kiefer-Wolfowitz inequality \citep{Ma} asserts that for any $z > 0$, 
\begin{align}
    \label{DKW}
    \mathbb{P}_{G_1}\left(\|F(\cdot,G_1)-F_n\|_\infty > z \right) \le2 \e^{-2n z ^2},
\end{align}
and consequently,
\[\E_{G_1}\|F_n-F(\cdot, G_1)\|_\infty\le \int_0^\infty 2 \e^{-2n z ^2}\dd z=\sqrt{\frac{\pi}{2}}n^{-1/2}.\]
The proof is complete.
\end{proof}

The following theorem is the key technical tool for the proof of Theorem~\ref{main}. We describe the main and novel ingredient in its proof, a coarse-graining tree, in Section~\ref{tree}.
\begin{thm}
    \label{orders}
   Let $G_0\in\mathcal{G}_{m_0}$. Under Assumption~B(2m),
   \begin{enumerate}[a.]
   \item  \label{local}   There are $\eps >0$ and $\delta>0 $ such that 
\begin{align*}
    \inf_{\substack{G_1,G_2\in\Glm \cap \mathcal{W}_{G_0}(\eps)\\G_1\ne G_2}}   \frac{\left\lVert F(\cdot, G_1) - F(\cdot, G_2) \right\rVert _\infty}{W(G_1, G_2)^{2m - 2m_0 + 1}} > \delta.
\end{align*}
\item \label{global} There exists $\delta > 0$ such that
\begin{equation*}
  \inf_{\substack{G_1,G_2\in\Glm\\G_1\ne G_2}} \frac{\left\lVert F(\cdot, G_1) - F(\cdot, G_2) \right\rVert _\infty}{W(G_1, G_2)^{2m-1}}  > \delta. 
\end{equation*}
   \end{enumerate}
\end{thm}
\begin{proof}[Proof of Theorem~\ref{main}]
Let $\eps >0$ as in Theorem~\ref{orders}.\ref{local} and set 
\[ z_\eps  = \sup_{G_1} \inf_{G_2} \|F(\cdot,G_1)-F(\cdot,G_2)\|_\infty  \]
where the supremum is taken over all $G_1\in \Glm\cap \mathcal{W}_{G_0}(\eps/2)$ and the infimum is taken over all $G_2\in \Glm\setminus \mathcal{W}_{G_0}(\eps)$. By compactness of $\Glm\setminus \mathcal{W}_{G_0}(\eps)$ the infimum is attained and by identifiability (coming from Assumption~B(2m)), it is nonzero. Thus, a fortiori, we have $z_\eps>0$.

Set $A_\eps=\{\|F(\cdot,G_1)-F_n\|_\infty < z_\eps /4\}$ ; on $A_\eps$, by  inequality \eqref{fois2} for the minimum distance estimator $\widehat{G} _n$,  we see that $\|F(\cdot,G_1)-F(\cdot, \widehat{G} _n)\|_\infty < z_\eps /2$. Hence, if $G_1\in \Glm\cap \mathcal{W}_{G_0}(\eps/2)$, then 
  $\widehat{G}_n$ is in $\mathcal{W} _{G_0}(\eps)$ and we may use Theorem~\ref{orders}.\ref{local}:
\begin{equation*}
  W(\widehat{G}_n,G_1)^{2m-2m_0+1} < \frac2\delta\, \|F_n-F(\cdot,G_1)\|_\infty .
\end{equation*}
Applying Lemma~\ref{lemDKW} with $A=A_\eps$, $C=2/\delta$ and $d=2m-2m_0+1$ together with \eqref{DKW} with $z=z_\eps/2$ yields 
\[\E_{G_1}W(\widehat{G}_n,G_1)\underset{\delta,d}{\lec} n^{-1/2d}\]
and bound~\eqref{main1} is proved. 

Applying the same Lemma~\ref{lemDKW} with $d=2m-1$  and Theorem~\ref{orders}.\ref{global}
likewise yields bound~\eqref{main2}.

\end{proof}

We now give two related results under weaker derivability assumptions, but less general. The first is the valid weaker version of Lemma~2 by \citet{Chen}, which is sufficient for the use other authors have made of it. Here, we only compare mixtures in a ball with the mixture at the center of the ball. The second covers the case where the number of components in the mixture is known, and is used for the proof of Theorem~\ref{choose_and_estimate}.
\begin{thm}
    \label{weak_Chen}
     Let $G_0\in \Glm$. Under Assumption~B(2),
 there are $\eps >0$ and $\delta>0 $ such that 
\begin{align*}
    \inf_{G_1\in\Glm \cap \mathcal{W}_{G_0}(\eps)}   \frac{\left\lVert F(\cdot, G_1) - F(\cdot, G_0) \right\rVert _\infty}{W(G_1, G_0)^{2}} > \delta.
\end{align*}
\end{thm}
\begin{proof}
We can follow the proof of \citet[Lemma 2]{Chen} which holds here, because the $\gamma _j$ defined in his paper are all non-negative, and at least one is nonzero. 
\end{proof}

\begin{thm}
    \label{equal_compo} Let $G_0\in\Gm{}$. Under Assumption~B(1), 
 there are $\eps >0$ and $\delta>0 $ such that 
\begin{align*}
    \inf_{\substack{G_1,G_2\in\Glm \cap \mathcal{W}_{G_0}(\eps)\\G_1\ne G_2}}   \frac{\left\lVert F(\cdot, G_1) - F(\cdot, G_2) \right\rVert _\infty}{W(G_1, G_2)} > \delta.
\end{align*}
\end{thm}
 The proof is given in the supplemental
part \citep{Supp}.

\begin{proof}[Proof of Theorem~\ref{choose_and_estimate}]
 Consider a fixed mixing distribution $G_0$ with exactly $m_0$ components. Set 
\begin{eqnarray*}
 \eps'&=&\inf_{G_1 \in \mathcal{G} _{< m_0}} \left\|F(\cdot, G_1) - F(\cdot, G_0) \right\|_{\infty} \\
\eps'' &=&\inf_{\substack{G_1 \in \Glmo \\ W(G_1,G_0)\ge \eps}} \left\| F(\cdot, G_1) - F(\cdot, G_0) \right\| _{\infty}.
\end{eqnarray*}
By compactness and identifiability, $\eps'$  and $\eps''$ are attained and positive. Let the event $A_n=\{\|F(\cdot,G_0)-F_n\|_{\infty} \le z_n\}$ with
\[z_n=\frac14\left[n^{-1/2+\kappa}\wedge(\eps'-n^{-1/2+\kappa})\wedge \eps''\right].\]
We first bound $W(\widehat{G}_n,G_0)$ on the event $A_n$: we have, by definition~\eqref{gn} of the minimum distance estimator $\widehat{G}_{n,m_0} $ on $\Glmo$,
\[ \| F(\cdot, \widehat{G}_{n,m_0}) - F_n \|_{\infty} \le \| F(\cdot, G_0) - F_n \|_{\infty} \le z_n \le n^{-1/2+\kappa},\]
so that $\hat{m}$ is at most $m_0$ ; moreover by the triangle's inequality, for all $G_1\in \mathcal{G}_{< m_0}$, 
\[\| F(\cdot, G_1) - F_n \|_{\infty}\ge \eps'-z_n> n^{-1/2+\kappa},\]
so that $\hat{m}$ is at least $m_0$ and thus $\widehat{G}_n=\widehat{G}_{n,m_0}\in \Glmo$. Moreover, 
\[\| F(\cdot, \widehat{G}_n) - F(\cdot,G_0) \|_{\infty} \le 2 \| F(\cdot, G_0) - F_n \|_{\infty}\le 2z_n<\eps'',\]
so that $\widehat{G}_n$ must be in $\mathcal{W}_{G_0}(\eps)$. Hence, by Theorem~\ref{equal_compo} and \eqref{fois2}, we get:
\begin{eqnarray*}
 W(\widehat{G}_n,G_0) &\le & \frac1\delta  \|F(\cdot,\widehat{G}_n)-F(\cdot,G_0)\|_{\infty}\\
 &\le & \frac2\delta \|F_n-F(\cdot,G_0)\|_{\infty}.
\end{eqnarray*}
By Lemma~\ref{lemDKW} for $A=A_n$, $C=2/\delta$ and $d=1$ and \eqref{DKW} for $z=z_n$, we deduce 
\[\E_{G_0}\left[ W(\widehat{G}_n,G_0)\right]\le  \frac2\delta\sqrt{\frac{\pi}{2}} n^{-1/2} + 2\Diam(\Theta)\e^{-2n^{2\kappa}},\]
and we are done.
\end{proof}

\section{Proofs}
\label{proofs}
\subsection{The coarse-graining tree and Theorem~\ref{orders}}
\label{tree}

Theorem~\ref{orders}.\ref{global} is a consequence of Theorem~\ref{orders}.\ref{local} and compactness and identifiability. Details in the supplemental part \citep{Supp}.

We  split the proof of Theorem~\ref{orders}.\ref{local} into three steps. 

\subsubsection*{Step 1: selecting  $(G_{1,n},G_{2,n})$ and related
  scaling sequences} We have to prove
\[\lim_{n\to \infty}\uparrow \inf_{\substack{G_1,G_2\in\Glm \cap
    \mathcal{W}_{G_0}(\frac1n)\\G_1\ne G_2} }   \frac{\left\lVert F(\cdot,
    G_1) - F(\cdot, G_2) \right\rVert _{\infty}}{W(G_1, G_2)^{2m - 2m_0 +
    1}} > \delta,\]
for some $\delta >0$. Choose for each $n$ distinct mixing distributions  $G_{1,n}, G_{2,n}$ in
$\Glm \cap\mathcal{W}_{G_0}(\frac1n)$  such that, setting $\Delta G_n=G_{1,n}- G_{2,n}$,
\[\inf_{\substack{G_1,G_2\in\Glm \cap\mathcal{W}_{G_0}(\frac1n)\\G_1\ne G_2} }   \frac{\left\lVert F(\cdot,
    G_1) - F(\cdot, G_2) \right\rVert _{\infty}}{W(G_1, G_2)^{2m - 2m_0
    +1}} +\frac{1}{n}\ge 
\frac{\left\| F(\cdot, \Delta G_n) \right\| _{\infty}}{W(\Delta G_n)^{2m -
    2m_0 + 1}}.\]
We may and do assume that $(G_{1,n})\subset \mathcal{G}
_{m_{1}}$ and $(G_{2,n})\subset \mathcal{G}
_{m_{2}}$ for some $m_1,m_2$ at most $m$.  We can then write
\[G_{1,n} =\sum_{j=1}^{m_1}\pi_{1,j,n}\delta_{\theta_{1,j,n}}
   \quad\text{and}\quad G_{2,n} =\sum_{j=m_1+1}^{m_1+m_2}\pi_{2,j,n}\delta_{\theta_{2,j,n}}\]
and thus the signed measure $\Delta G_n$ is:
\begin{equation*}
 \Delta G_n=\sum_{j=1}^{m_1+m_2}\pi_{j,n}\delta_{\theta_{j,n}}
\end{equation*}
with 
\[(\pi_{j,n} ,\theta_{j,n})=
\begin{cases}
(\pi_{1,j,n},\theta_{1,j,n}) & \text{for $j\in\lb 1,m_1\rb$ }\\
(- \pi_{2,j,n}, \theta_{2,j,n}) & \text{for $j\in\lb m_1+1,m_2\rb$ }
\end{cases}.
\]
Up to selecting a subsequence of $\Delta G_n$,
we may find a finite number of scaling sequences
$\eps_s(n)$, $s\in \lb 0,S\rb$,  such that 
\begin{equation}
  \label{eps_scale}
  0\equiv \eps_0(n) < \eps_1(n)  < \cdots < \eps_S(n)  \equiv
  1\text{ with }\eps_s(n)  = o\big(\eps_{s+1}(n)\big) ,
\end{equation}
and  such that they are of the same order as  the rates of
convergence of the various $|\theta_{j,n} - \theta_{j',n}|$ for $j,j'\in \lb
1,m_1+m_2\rb$ and
$\lvert \sum_{j\in J}\pi_{j,n}\rvert$ for $J\subset \lb
1,m_1+m_2\rb$. That is, there are integers  $\fs(j,j')$  and $\fs_\pi(J)$  in $\lb 0,S\rb$  
such that
\begin{equation}
  \label{sjjsJ}
 \left\lvert \theta_{j,n} - \theta_{j',n}
 \right\rvert \asymp \eps_{\fs(j,j')}(n)\quad\text{and}\quad \left\lvert \sum_{j\in J}\pi_{j,n}
 \right\rvert \asymp \eps_{\fs_\pi(J)}(n). 
\end{equation}
Note that the map $\fs(\cdot,\cdot)$ defined by \eqref{sjjsJ} is an ultrametric on $\lb 1,m_1+m_2\rb$ (but does not
separate points).  We also define  the $\fs$-diameter of subsets $J$ of $\lb 1,m_1+m_2\rb$
by
\begin{align*}
    \fs(J) & = \max_{j,j' \in J} \fs(j,j').   
\end{align*}

\subsubsection*{Step 2: construction of the coarse-graining tree and key lemmas} 
Consider the collection $\T$ of distinct ultrametric balls
$J=\{j':\fs(j',j)\le s\}$  that we can make when $j$ ranges over $\lb 1,m_1+m_2\rb$
and $s$ over $\lb 0,S\rb$. This collection defines the coarse-graining
tree we need. Its root is $J_o=\lb 1,m_1+m_2\rb$ and its
nodes $J$ satisfy 
\[J\cap J'\ne \emptyset\implies J\subset J' \text{ or } J'\subset J,\]
by the ultrametric property. 

Let us show how the tree $\T$ looks like with a partial representation :  
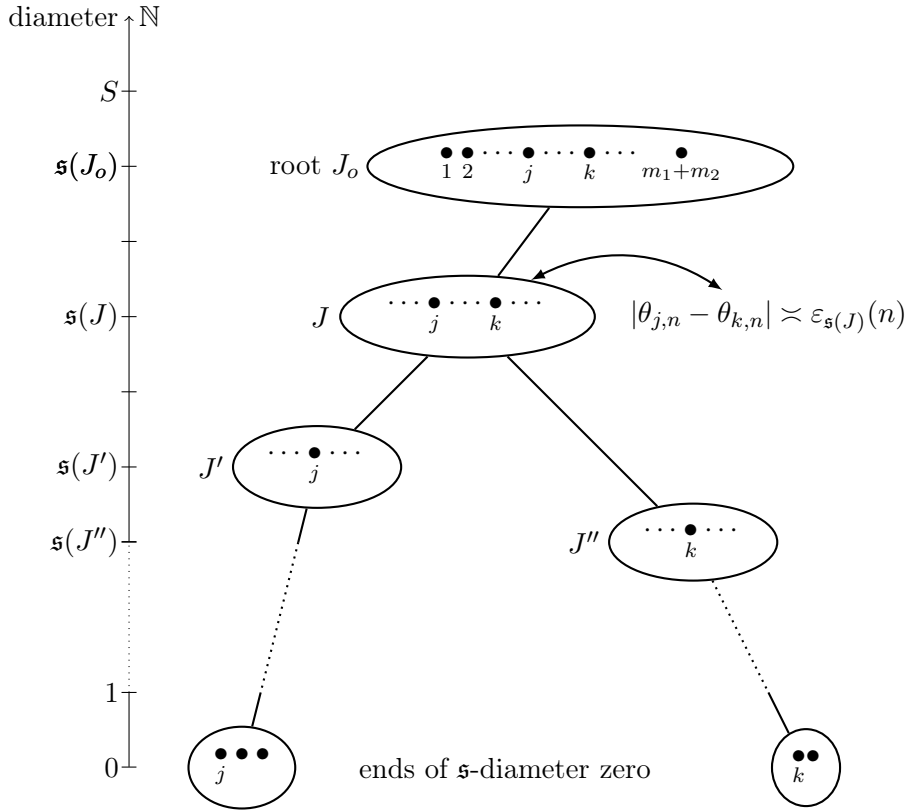
\begin{center}
\begin{tikzpicture}
  \tikzstyle{noeud}=[ellipse,draw,thick]
\tikzstyle{noeudpoint}=[ellipse,draw,thick,dashed]
 \tikzstyle{arete}=[-,thick]
\tikzstyle{aretepoint}=[-,thick,dotted]
\tikzstyle{corres}=[<->,>=latex,thick]

\draw [|->] (-6,-5) -- (-6,2);
\draw [|-] (-6,-8) -- (-6,-7);
\draw[dotted] (-6,-7)--(-6,-5);
\draw (-6,2) node[left] {diameter};
\draw (-6,2) node[right] {$\N$};
\draw (-6,1) node[left] {$S$};
\draw (-6,0) node[left] {$\fs(J_o)$};
\draw (-6,0) node[left] {$\fs(J_o)$};
\draw (-6,-2) node[left] {$\fs(J)$};
\draw (-6,-4) node[left] {$\fs(J')$};
\draw (-6,-5) node[left] {$\fs(J'')$};
\draw (-6,-7) node[left] {1};
\draw (-6,-8) node[left] {0};
\draw (-6.1,-7) -- (-5.9,-7);
\draw (-6.1,-5) -- (-5.9,-5);
\draw (-6.1,-4) -- (-5.9,-4);
\draw (-6.1,-3) -- (-5.9,-3);
\draw (-6.1,-2) -- (-5.9,-2);
\draw (-6.1,-1) -- (-5.9,-1);
\draw (-6.1,0) -- (-5.9,0);
\draw (-6.1,1) -- (-5.9,1);
  \node[noeud] (R) at (0,0) {$\underset{1}{\bullet}\underset{2}{\bullet}\cdots \underset{j}{\bullet}\cdots \underset{k}{\bullet}\cdots\underset{m_1+m_2}{\bullet}$};
\draw[left] (-2.8,0) node {root $J_o$} ;

\node[noeud] (J) at (-1.5,-2) {$\cdots \underset{j}{\bullet}\cdots\underset{k}{\bullet}\cdots$};
\draw[left] (-3.2,-2) node {$J$} ;

\node (textJ) at (2.5,-2){$|\theta_{j,n}-\theta_{k,n}|\asymp\eps_{\fs(J)}(n)$};
\draw[corres] (J)to[bend left](textJ);

\node[noeud] (J') at (-3.5,-4) {$\cdots\underset{j}{\bullet}\cdots$};
  \draw[left] (-4.6,-4) node {$J'$} ;

\node[noeud] (J'') at (1.5,-5) {$\cdots\underset{k}{\bullet}\cdots$};
  \draw[left] (0.4,-5) node {$J''$} ;

\node[noeud] (end) at (-4.5,-8) {$\underset{j}{\bullet}\bullet\bullet$};

\node[noeud] (end2) at (3,-8) {$\underset{k}{\bullet} \bullet$};

\node (textend2) at (-1,-8){ends of $\fs$-diameter zero};

\draw[arete] (R)--(J);
\draw[arete] (J)--(J');
\draw[arete] (J)--(J'');
\draw[aretepoint] (J'')--(2.5,-7);
\draw[arete] (2.5,-7)--(end2);
\draw[arete] (J')--(-3.75,-5);
\draw[aretepoint] (-3.75,-5)--(-4.25,-7);
\draw[arete] (-4.25,-7)--(end);
\end{tikzpicture}
\end{center}
\medskip 

Note that the ends are not necessarily singletons since the metric $\fs(\cdot,\cdot)$ does not
separate points.

We define the \emph{parent} $\parent{J}$
of a node  $J\subsetneq J_0$ by  
\[(J\subset I\subsetneq \parent{J},
I\in\T)\implies I=J.\]
The \emph{set of
descendants} and the \emph{set of children} of a node $J$ are
\begin{eqnarray*}
 \mathrm{Desc}(J) &=& \{I\in\T : \parent{I}\subset J\}, \\
\mathrm{Child}(J) & =& \{I\in\T :\parent{I}= J\}.
\end{eqnarray*}
The following two lemmas are proved in the supplement part
\citep[Section~3]{Supp}. 
\begin{lem}\label{scW} With the above notations,
  \begin{equation}
\label{scaleW}
    W\left(\Delta G_n\right)  \asymp\max_{J\in \mathrm{Desc}(J_o)} \eps_{\fs_\pi(J)}(n) \eps_{\fs(\parent{J})}(n).
\end{equation}
\end{lem}
Set now  for $J\subset J_o$, 
\[F(x,J)=\sum_{j\in J}\pi_{j,n}F(x,\theta_{j,n}),\]
so that, in particular, $F(x,\Delta G_n)=F(x,J_o)$.  We shall use  Taylor expansions along the tree $\T$ to express the order of $F(x,\Delta G_n)$ in terms of the scaling functions $\eps_s(n)$.

\begin{lem}\label{lemrec}
  Let $J$ be a node and set $d_J=\mathrm{card}(J)$. Pick $\theta_J$ in
  the set $\{\theta_{j,n}:j\in J\}$. The dependence on  $n$ is  skipped from the following notations. There are a vector $a_J=(a_J(k))_{0\le k\le 2m}$ and a remainder $R(x,J)$ such that 
\begin{equation}
    \label{hyprec}
    F(x,J) = \sum_{k=0}^{2m} a_J(k) \eps_{\fs(J)}^k F^{(k)}(x, \theta _J) + R(x,J),
\end{equation}
where:
\begin{enumerate}[(a)]
\item \label{coeffsbornes}$\displaystyle a_J(0)=\sum_{j\in
    J}\pi_j$  and $|a_J(k)|\lec 1 $ for all $k\le 2m$,
 \item \label{premierscoeffs} There is a coefficient
   $a_{J}(k)$ of maximal order among the $d_J$ first
   ones. That is, there is an integer $k_J<d_J$ such that 
            \begin{equation*}
            \|a_J \|= \max_{k\le 2m} | a_{J}(k) |
            \asymp  |a_J(k_J)| ,
        \end{equation*}
\item  \label{borneinfcoeff}  The norm $\|a_J \|$ is bounded from below (up to a constant) by
  a quantity linked to the Wasserstein distance: 
\[\|a_J \|\gec \max\left(\eps _{\fs_\pi(J)},
  \max_{I\in\mathrm{Desc}(J)}\eps _{\fs_\pi(I)} \left(\frac{\eps
      _{\fs(\parent{I})}}{\eps_{\fs(J)}}\right) ^{d_J - 1}\right),\]
\item \label{remaind}  The remainder term is negligible. Uniformly in $x$:
 \[R(x,J) = o\left(\|a_J\|\, \varepsilon_{\fs(J)}^{2m}\right),\]
\item \label{epsep} For distinct $I,I'\in \C{J}$, we have
  $|\theta_{I}-\theta_{I'}|\asymp \eps_{\fs(J)}$.
\end{enumerate}  
\end{lem}

\subsubsection*{Step 3: concluding the proof of Theorem~\ref{orders}.\ref{local}}  Consider the root $J_o$
of the tree $\T$ and distinguish two cases:

\noindent\textbf{Case 1: $\fs(J_o)<S$.} Set for short $J=J_o$. In this case we have
$\eps_{\fs(J)}=o(1)$ and may apply directly Lemma~\ref{lemrec} to
$J$:
\[F(x,\Delta G_n) = F(x,J)=\sum_{k=0}^{2m} a_{J}(k) \eps _{\fs(J)}^k
F^{(k)}(x, \theta _{J}) + R(x,J) , \]
so that 
\begin{equation}
  \label{eq:triang}
\left\lVert F(\cdot,\Delta G_n) \right\rVert _{\infty}\ge   \left\|\sum_{k=0}^{2m} a_{J}(k) \eps
  _{\fs(J)}^k F^{(k)}(\cdot, \theta _{J})\right\|_\infty - \left\|R(\cdot, \theta _{J})\right\|_\infty.
\end{equation}
By using the lower bound \eqref{alpha_bound_sep}, we get for all $k<d_J$
\begin{equation}
  \label{eq:minorm}
 \left\|\sum_{k=0}^{2m} a_{J}(k) \eps
  _{\fs(J)}^k F^{(k)}(\cdot, \theta _{J})\right\|_\infty\gec \max_k
\left|a_{J}(k)\eps _{\fs(J)}^k\right|\ge \left|a_{J}(k)\eps _{\fs(J)}^{d_J-1}\right| ,
\end{equation}
and, taking $k=k_{J}$,  (\ref{premierscoeffs}) and
(\ref{borneinfcoeff})  yield
\begin{equation*}
  |a_{J}(k)\eps _{\fs(J)}^k|\gec
  \max_{I\in\D{J}}\eps_{\fs_\pi(I)}\eps_{\fs(\parent{I})}^{d_{J}-1};
\end{equation*}
going on, since $d_J \le m_1+m_2\le 2m$, we get 
\begin{equation}
  \label{eq:mincoeff}
  |a_{J}(k)\eps _{\fs(J)}^k|\gec 
  \max_{I\in\D{J}}\eps_{\fs_\pi(I)}\eps_{\fs(\parent{I})}^{2m-1} .
\end{equation}
Since $R(x,J)$ is of smaller order by (\ref{remaind}), we get from
\eqref{eq:triang}, \eqref{eq:minorm}, \eqref{eq:mincoeff} 
\[\left\lVert F(\cdot,\Delta G_n) \right\rVert _{\infty}\gec  \max_{I\in\D{J}}\eps_{\fs_\pi(I)}
\eps_{\fs(\parent{I})}^{2m-1}\]
so that Lemma~\ref{scW} gives
\[\left\lVert F(\cdot,\Delta G_n) \right\rVert _{\infty}\gec W(\Delta G_n)^ {2m-1}.\]
\textbf{Case 2:  $\fs(J_o)=S$.} We split $\Delta G_n$ over the
  first-generation children:
  \begin{eqnarray*}
  F(x,\Delta G_n) = F(x,J_o)&=&\sum_{J\in\C{J_o}}F(x,J)\\
&=& \sum_{J\in\C{J_o}}\left[\sum_{k=0}^{2m} a_J(k) \eps _{\fs(J)}^kF^{(k)}(x, \theta _J) + R(x,J)  \right].
  \end{eqnarray*}
Moreover, by  (\ref{epsep}), the $\theta _J$ for
$J\in\C{J_o}$ are $\eps$-separated for some $\eps >0$ so that the
lower bound  \eqref{alpha_bound_sep} can be applied in the bracket above and yields, since the
$R(x,J)$'s are negligible:
\begin{equation*}
 \left\| F(\cdot,\Delta G_n) \right\| _{\infty}\gec
  \max_{J\in\C{J_o}}\max_{k\le 2m}|a_{J}(k)\eps _{\fs(J)}^k|\ge
  \max_{J\in\C{J_o}}\max_{k<d_J}|a_{J}(k)\eps _{\fs(J)}^k|. 
\end{equation*}
If we take $k=0$ rather than the maximum for $k<d_J$ in the last bound above, we deduce from (\ref{coeffsbornes}) that  
\[\left\lVert F(\cdot,\Delta G_n) \right\rVert _{\infty}\gec
\max_{J\in\C{J_o}}\eps_{\fs_\pi(J)},\] 
whereas if we substitute $\eps _{\fs(J)}^{d_J-1}$ to $\eps
_{\fs(J)}^k$, we get from  (\ref{premierscoeffs}) and next (\ref{borneinfcoeff})
\begin{eqnarray*}
  \left\lVert F(\cdot,\Delta G_n) \right\rVert _{\infty} &\gec&
  \max_{J\in\C{J_o}}\|a_J\|\eps _{\fs(J)}^{d_J-1}\\
&\gec&
  \max_{J\in\C{J_o}}\max_{I\in\D{J}}\eps_{\fs_\pi(I)}
\eps_{\fs(\parent{I})}^{d_J-1}.
\end{eqnarray*}
We may combine the two lower bounds above and, after recalling that
$\eps_{\fs(J_o)}=1$ and setting $d_\star=\max_{J\in\C{J_o}}d_J$,  get
\begin{eqnarray*}
  \left\lVert F(\cdot,\Delta G_n) \right\rVert _{\infty}&\gec&  \max_{J\in\C{J_o}}\max_{I\in\D{J}\cup\{J\}}\eps_{\fs_\pi(I)}
\eps_{\fs(\parent{I})}^{d_J-1}\\
&\gec&
\max_{J\in\D{J_o}}\eps_{\fs_\pi(J)}\eps_{\fs(\parent{J})}^{d_\star-1}\\
&\gec& W(\Delta G_n)^ {d_\star-1} ,
\end{eqnarray*}
where the last inequality comes from Lemma~\ref{scW}. It remains to
estimate $d_\star$. Since $G_{1,n}$ and $G_{2,n}$ converge to $G_0\in
\Gm0$, the root $J_o$ (of cardinality $m_1+m_2$) has at least $m_0$
children with at least two elements. Thus, the cardinality $d_\star$
of the biggest child is bounded by $m_1+m_2-2(m_0-1)$ so that 
\[ \left\lVert F(\cdot,G_n) \right\rVert _{\infty}\gec W(\Delta G_n)^ {m_1+m_2-2m_0+1}\gec W(\Delta G_n)^ {2m-2m_0+1}.\]

Finally, if $m_0$ is more than one, we are in Case 2 where
$\fs(J_o)=S$ and if $m_0$ is one, Case 1 and Case 2 can occur. But whatever the case, we always have 
\[\left\lVert F(\cdot,G_n) \right\rVert _{\infty}\gec W(\Delta G_n)^ {2m-2m_0+1}\]
so that Theorem~\ref{orders}.\ref{local}.  is proved.

\subsection{Proof of  Theorem~\ref{LAN}}
 Set $d=m-m_0+1$ for short. 
 Consider  numbers $\mu_0=1, \mu_1, \dots,
 \mu_{2d-2}$ such that the Hankel matrices  $(M_k)_{i,j} =
 \mu_{i+j-2}$ satisfy $\det M_k > 0$ for $k\in\lb 1,d-1\rb$. By
 Theorem~\ref{Lindsay}.\ref{lind}., we may then define for any real number $u$
 a distribution $G(u) = \sum_{j= m_0}^m \pi_j(u) \delta _{h_j(u)}$
with initial moments $1, \mu_1, \dots, \mu_{2d-2}, \mu_{2d-1} =u$. Moreover, the unicity in Theorem \ref{Lindsay}.\ref{lind}. implies that, on the
set 
\[\left\{(\pi_1,\ldots,\pi_d,h_1,\ldots,h_d)\in\R^{2d}:\pi_1>0,\ldots,\pi_d>0,h_1<\cdots
<h_d\right\},\]
the following application
is injective:
\begin{align*}
\phi :(\pi_1,\ldots,\pi_d,h_1,\ldots,h_d)\mapsto 
\left(\sum_1^d\pi_j,\sum_1^d\pi_j h_j,\sum_1^d\pi_j h_j^2,\ldots,\sum_1^d\pi_j h_j^{2d-1}\right).
\end{align*}
Now, its Jacobian is non-zero \citep[see][Section~\ref{jaco}]{Supp}: 
\begin{align}
    \label{Jac}
J(\phi)=(-1)^{\frac{(d-1)d}{2}}\,\pi_1\cdots\pi_d \prod_{1\le j<k\le d}(h_j- h_k)^4.  
\end{align}
Thus the inverse of $\phi$ is locally continuous, so that, in
particular,  the $h_j(u)$
are all continuous. Thus, we can set $H(U)=\max_{j\le d}\max_{|u|\le
  U}\left|h_j(u)\right|$ which is finite for any $U>0$ and we choose a
positive sequence $U(n)$ such that $U(n) \to \infty$ and $H(U(n)) n^{-1/(4d-2)} \to 0$. 

We now define support points $\theta _{j,n}(u)  = \theta _{m_0} +
n^{-1/(4d-2)} h_j(u) $ in $\Theta$ and mixing distributions around
$G_0$ by 
\begin{equation}
 G_n(u)  = \sum_{j=1}^{m_0 - 1} \pi_j \delta_{\theta _j} + \pi_{m_0}
\sum_{j= m_0}^m \pi_j(u) \delta _{\theta_{j,n}(u)}.
\end{equation}
Note that $G_n(0)$ and $G_0$ do not coincide.
The form of $G_n(u)$ makes it clear that it converges to $G_0$ at
speed  $n^{-1/(4d - 2)} $: it is easily seen  that for  $|u|\le U$
\[W(G_{n}(u), G_0)\le \pi_{m_0}H(U) n^{-1/(4d - 2)}.\]
This proves \eqref{lan1}.

Moreover, since all other points and proportions are equal, the
transportation distance  $ W(G_{n}(u), G_{n}(u'))$ is equal to the
transportation distance between the last $p$ components. Since those
support points keep the same weights and are homothetic with scale
$n^{-1/(4d- 2)}$ around $\theta _{m_0}$, we have exactly 
\[ W(G_{n}(u), G_{n}(u')) =  W(G_{1}(u), G_{1}(u')) n^{-1/(4d - 2)}.\]
This proves \eqref{lan2}.

We now prove local asymptotic normality. As before, the probability under the mixing distribution $G_n(0)$ is denoted by $\P_{G_n(0)}$ and the corresponding expectation
$\E_{G_n(0)}$. Let $X_{1,n}, \dots, X_{n,n}$ be an
i.i.d. sample with density $\otimes_{i=1}^nf\left(\cdot,G_n(0)\right)$. Then, we can write the log-likelihood ratio as
\begin{equation*}
    Z_{n,0}(u)=\ln \left(\frac{\prod_{i=1}^nf(X_{i,n}, G_n(u))}{\prod_{i=1}^nf(X_{i,n}, G_n(0))}\right)=\sum_{i=1}^n \ln \left( 1 + Y_{i,n} (u)\right)
\end{equation*}
with
\begin{equation}
\label{Yinu}
  Y_{i,n} (u)= \frac{f\left(X_{i,n}, G_n(u)\right)-f\left(X_{i,n}, G_n(0)\right) }{ f(X_{i,n}, G_n(0))}.
\end{equation}
Set also
\begin{equation}
\label{Zin}
Z_{i,n}=\frac{f^{(2d-1)}(X_{i,n},\theta_{m_0})}{f(X_{i,n}, G_n(0))}.  
\end{equation}
Using Taylor expansions with remainder, we find that $Y_{i,n} (u)$ and $Z_{i,n}$ are centered under $\P_{G_n(0)}$ \citep[see][]{Supp}.

Consider now
\begin{equation}
\label{Zn}
  Z_n=\pi_{m_0}n^{-1/2}\sum_{i=1}^nZ_{i,n}.
\end{equation}
By Proposition~\ref{tversmix} in the supplemental part \citep{Supp},  for $n$ large enough, we have $\E_{G_n(0)}
\left|Z_{1,n}\right|^2 \asymp 1 $. Up to taking a subsequence, we may
then assume  $\E_{G_n(0)} \left|Z_{1,n}\right|^2\to \sigma ^2$ for some
positive $\sigma $. By Proposition~\ref{tversmix} again, we have $
\E_{G_n(0)} \left|Z_{1,n}\right|^3\lec 1$ for all $n$ large enough.

We may then apply Lyapunov theorem \cite[Theorem 23.7]{Bill} to prove that
\begin{align}
    \label{cvZn}
    Z_n &\xrightarrow[]{d} \mathcal{N} (0, \Gamma
          )\quad\text{with}\quad    \Gamma = \sigma ^2\pi_{m_0}^2.
\end{align}
Indeed, the Lyapunov condition holds:
\[\frac{\sum_{i=1}^n \E_{G_n(0)} \left|Z_{i,n}\right|^3}{\left[\sum_{i=1}^n\E_{G_n(0)}
\left|Z_{i,n}\right|^2\right]^{3/2}}\lec \frac{1}{\sigma^3\sqrt{n}}\xrightarrow[n\to\infty]{} 0\]
so that 
\[\frac{\sum_{i=1}^n Z_{i,n}}{\sqrt{\sum_{i=1}^n\E_{G_n(0)}
\left|Z_{i,n}\right|^2}}=\frac{\sum_{i=1}^n Z_{i,n}}{\sqrt{n\E_{G_n(0)}
\left|Z_{1,n}\right|^2}}\xrightarrow[]{d} \mathcal{N} (0, 1 )\]
and \eqref{cvZn} follows easily from \eqref{Zn}. 

Now, to get the convergence  in probability of
$Z_{n,0}-uZ_n+\frac{u^2}{2}\Gamma$ to zero, we show in the supplemental part \citep{Supp} the following convergences for all $u$:
\begin{eqnarray}
A_n(u)=\sum_{i=1}^nY_{i,n}(u) -uZ_n & \xrightarrow[]{L^2} & 0, \label{p1}\\
B_n(u)=  \sum_{i=1}^nY_{i,n}(u)^2-u^2\Gamma& \xrightarrow[]{L^1}  &0, \label{p2}\\
 C_n(u)=\sum_{i=1}^n|Y_{i,n}(u)|^3&  \xrightarrow[]{L^1}& 0. \label{p3}
\end{eqnarray}
Then, setting 
\[D_n(u)=Z_{n,0}(u)-\sum_{i=1}^nY_{i,n}(u)+\frac12
  \sum_{i=1}^nY_{i,n}(u)^2,\]
we  have, since $|\ln (1+y)-y+y^2/2|\le |y|^3$ for $|y|\le
2/3$,
\[\left|D_n(u)\right|\le C_n(u)\]
with probability going to one, so that
\begin{equation*}
  Z_{n,0}(u)-uZ_n+\frac{u^2}{2}\Gamma = A_n(u)+\frac12 B_n(u)+D_n(u)
\end{equation*}
tend to $0$ in probability.

\bigskip

\textbf{Acknowledgements.} We thank Sébastien Gadat for numerous suggestions that have greatly improved the presentation of the paper, and Élisabeth Gassiat for helpful discussions.
\medskip

This work was supported in part by the Labex CEMPI  (ANR-11-LABX-0007-01).

\begin{supplement}[id=suppA]
\stitle{Auxiliary results and technical details}
\slink[doi]{10.1214/00-AOASXXXXSUPP}
\sdatatype{.pdf} 
\sdescription{This supplemental part gathers some proof details on some assertions given in the paper.}
\end{supplement}

\bibliographystyle{imsart-nameyear}
\bibliography{minimax_mixtures_final}

\begin{thebibliography}{34}

\bibitem[\protect\citeauthoryear{Billingsley}{1995}]{Bill}
\begin{bbook}[author]
\bauthor{\bsnm{Billingsley},~\bfnm{P.}\binits{P.}}
(\byear{1995}).
\btitle{Probability and Measure}.
\bseries{Wiley Series in Probability and Statistics}.
\bpublisher{Wiley}, \baddress{New York}.
\end{bbook}
\endbibitem

\bibitem[\protect\citeauthoryear{Bontemps and Gadat}{2014}]{MR3263130}
\begin{barticle}[author]
\bauthor{\bsnm{Bontemps},~\bfnm{Dominique}\binits{D.}} \AND
  \bauthor{\bsnm{Gadat},~\bfnm{S{\'e}bastien}\binits{S.}}
(\byear{2014}).
\btitle{Bayesian methods for the shape invariant model}.
\bjournal{Electron. J. Stat.}
\bvolume{8}
\bpages{1522--1568}.
\end{barticle}
\endbibitem

\bibitem[\protect\citeauthoryear{Caillerie et~al.}{2013}]{CCDM}
\begin{bincollection}[author]
\bauthor{\bsnm{Caillerie},~\bfnm{Claire}\binits{C.}},
  \bauthor{\bsnm{Chazal},~\bfnm{Fr{\'e}d{\'e}ric}\binits{F.}},
  \bauthor{\bsnm{Dedecker},~\bfnm{J{\'e}r{\^o}me}\binits{J.}} \AND
  \bauthor{\bsnm{Michel},~\bfnm{Bertrand}\binits{B.}}
(\byear{2013}).
\btitle{Deconvolution for the {W}asserstein metric and geometric inference}.
In \bbooktitle{Geometric Science of Information}
\bpages{561--568}.
\bpublisher{Springer}.
\end{bincollection}
\endbibitem

\bibitem[\protect\citeauthoryear{Chen}{1995}]{Chen}
\begin{barticle}[author]
\bauthor{\bsnm{Chen},~\bfnm{J.}\binits{J.}}
(\byear{1995}).
\btitle{Optimal {R}ate of {C}onvergence for {F}inite {M}ixture {M}odels}.
\bjournal{The Annals of Statistics}
\bvolume{23}
\bpages{221-233}.
\end{barticle}
\endbibitem

\bibitem[\protect\citeauthoryear{Dacunha-Castelle and Gassiat}{1997}]{Gass}
\begin{barticle}[author]
\bauthor{\bsnm{Dacunha-Castelle},~\bfnm{Didier}\binits{D.}} \AND
  \bauthor{\bsnm{Gassiat},~\bfnm{\'Elisabeth}\binits{E.}}
(\byear{1997}).
\btitle{The estimation of the order of a mixture model}.
\bjournal{Bernoulli}
\bpages{279--299}.
\end{barticle}
\endbibitem

\bibitem[\protect\citeauthoryear{Dedecker and Michel}{2013}]{DedMic}
\begin{barticle}[author]
\bauthor{\bsnm{Dedecker},~\bfnm{J{\'e}r{\^o}me}\binits{J.}} \AND
  \bauthor{\bsnm{Michel},~\bfnm{Bertrand}\binits{B.}}
(\byear{2013}).
\btitle{Minimax rates of convergence for {W}asserstein deconvolution with
  supersmooth errors in any dimension}.
\bjournal{Journal of Multivariate Analysis}
\bvolume{122}
\bpages{278--291}.
\end{barticle}
\endbibitem

\bibitem[\protect\citeauthoryear{Deely and Kruse}{1968}]{Deely&Kruse}
\begin{barticle}[author]
\bauthor{\bsnm{Deely},~\bfnm{J.~J.}\binits{J.~J.}} \AND
  \bauthor{\bsnm{Kruse},~\bfnm{R.~L.}\binits{R.~L.}}
(\byear{1968}).
\btitle{Construction of Sequences Estimating the Mixing Distribution}.
\bjournal{The Annals of Mathematical Statistics}
\bvolume{39}
\bpages{286--288}.
\end{barticle}
\endbibitem

\bibitem[\protect\citeauthoryear{Dudley}{2002}]{Dudley}
\begin{bbook}[author]
\bauthor{\bsnm{Dudley},~\bfnm{R.~M.}\binits{R.~M.}}
(\byear{2002}).
\btitle{Real analysis and probability}.
\bseries{Cambridge Studies in Advanced Mathematics}
\bvolume{74}.
\bpublisher{Cambridge University Press, Cambridge}
\bnote{Revised reprint of the 1989 original}.
\end{bbook}
\endbibitem

\bibitem[\protect\citeauthoryear{Fan}{1991}]{Fan}
\begin{barticle}[author]
\bauthor{\bsnm{Fan},~\bfnm{Jianqing}\binits{J.}}
(\byear{1991}).
\btitle{On the optimal rates of convergence for nonparametric deconvolution
  problems}.
\bjournal{The Annals of Statistics}
\bpages{1257--1272}.
\end{barticle}
\endbibitem

\bibitem[\protect\citeauthoryear{Gassiat and van Handel}{2013}]{MR3015722}
\begin{barticle}[author]
\bauthor{\bsnm{Gassiat},~\bfnm{Elisabeth}\binits{E.}} \AND
  \bauthor{\bparticle{van} \bsnm{Handel},~\bfnm{Ramon}\binits{R.}}
(\byear{2013}).
\btitle{Consistent order estimation and minimal penalties}.
\bjournal{IEEE Trans. Inform. Theory}
\bvolume{59}
\bpages{1115--1128}.
\end{barticle}
\endbibitem

\bibitem[\protect\citeauthoryear{Genovese and Wasserman}{2000}]{MR1810921}
\begin{barticle}[author]
\bauthor{\bsnm{Genovese},~\bfnm{Christopher~R.}\binits{C.~R.}} \AND
  \bauthor{\bsnm{Wasserman},~\bfnm{Larry}\binits{L.}}
(\byear{2000}).
\btitle{Rates of convergence for the {G}aussian mixture sieve}.
\bjournal{Ann. Statist.}
\bvolume{28}
\bpages{1105--1127}.
\end{barticle}
\endbibitem

\bibitem[\protect\citeauthoryear{Ghosal and van~der Vaart}{2001}]{MR1873329}
\begin{barticle}[author]
\bauthor{\bsnm{Ghosal},~\bfnm{Subhashis}\binits{S.}} \AND
  \bauthor{\bparticle{van~der} \bsnm{Vaart},~\bfnm{Aad~W.}\binits{A.~W.}}
(\byear{2001}).
\btitle{Entropies and rates of convergence for maximum likelihood and {B}ayes
  estimation for mixtures of normal densities}.
\bjournal{Ann. Statist.}
\bvolume{29}
\bpages{1233--1263}.
\end{barticle}
\endbibitem

\bibitem[\protect\citeauthoryear{H{\'a}jek}{1972}]{hajek1972local}
\begin{binproceedings}[author]
\bauthor{\bsnm{H{\'a}jek},~\bfnm{Jaroslav}\binits{J.}}
(\byear{1972}).
\btitle{Local asymptotic minimax and admissibility in estimation}.
In \bbooktitle{Proceedings of the {S}ixth {B}erkeley {S}ymposium on
  {M}athematical {S}tatistics and {P}robability ({U}niv. {C}alifornia,
  {B}erkeley, {C}alif., 1970/1971), {V}ol. {I}: {T}heory of statistics}
\bpages{175--194}.
\bpublisher{Univ. California Press, Berkeley, Calif.}
\end{binproceedings}
\endbibitem

\bibitem[\protect\citeauthoryear{Heinrich and Kahn}{2015}]{Supp}
\begin{bmisc}[author]
\bauthor{\bsnm{Heinrich},~\bfnm{P.}\binits{P.}} \AND
  \bauthor{\bsnm{Kahn},~\bfnm{J.}\binits{J.}}
(\byear{2015}).
\btitle{Supplement to 'Optimal rates for finite mixture estimation': Auxiliary
  results and technical details}.
\end{bmisc}
\endbibitem

\bibitem[\protect\citeauthoryear{Holzmann, Munk and Stratmann}{2004}]{Hol}
\begin{barticle}[author]
\bauthor{\bsnm{Holzmann},~\bfnm{Hajo}\binits{H.}},
  \bauthor{\bsnm{Munk},~\bfnm{Axel}\binits{A.}} \AND
  \bauthor{\bsnm{Stratmann},~\bfnm{Bernd}\binits{B.}}
(\byear{2004}).
\btitle{Identifiability of finite mixtures-with applications to circular
  distributions}.
\bjournal{Sankhy{\=a}: The Indian Journal of Statistics}
\bpages{440--449}.
\end{barticle}
\endbibitem

\bibitem[\protect\citeauthoryear{Ishwaran, James and Sun}{2001}]{Ish}
\begin{barticle}[author]
\bauthor{\bsnm{Ishwaran},~\bfnm{H.}\binits{H.}},
  \bauthor{\bsnm{James},~\bfnm{L.~F.}\binits{L.~F.}} \AND
  \bauthor{\bsnm{Sun},~\bfnm{J.}\binits{J.}}
(\byear{2001}).
\btitle{Bayesian model selection in finite mixtures by marginal density
  decompositions}.
\bjournal{J. Amer. Statist. Assoc.}
\bvolume{96}
\bpages{1316--1332}.
\end{barticle}
\endbibitem

\bibitem[\protect\citeauthoryear{Kuhn et~al.}{2014}]{kuhn2014spatial}
\begin{barticle}[author]
\bauthor{\bsnm{Kuhn},~\bfnm{Michael~A}\binits{M.~A.}},
  \bauthor{\bsnm{Feigelson},~\bfnm{Eric~D}\binits{E.~D.}},
  \bauthor{\bsnm{Getman},~\bfnm{Konstantin~V}\binits{K.~V.}},
  \bauthor{\bsnm{Baddeley},~\bfnm{Adrian~J}\binits{A.~J.}},
  \bauthor{\bsnm{Broos},~\bfnm{Patrick~S}\binits{P.~S.}},
  \bauthor{\bsnm{Sills},~\bfnm{Alison}\binits{A.}},
  \bauthor{\bsnm{Bate},~\bfnm{Matthew~R}\binits{M.~R.}},
  \bauthor{\bsnm{Povich},~\bfnm{Matthew~S}\binits{M.~S.}},
  \bauthor{\bsnm{Luhman},~\bfnm{Kevin~L}\binits{K.~L.}},
  \bauthor{\bsnm{Busk},~\bfnm{Heather~A}\binits{H.~A.}} \betal{et~al.}
(\byear{2014}).
\btitle{The Spatial Structure of Young Stellar Clusters. I. Subclusters}.
\bjournal{The Astrophysical Journal}
\bvolume{787}
\bpages{107}.
\end{barticle}
\endbibitem

\bibitem[\protect\citeauthoryear{Le~Cam}{1960}]{LeCam1}
\begin{barticle}[author]
\bauthor{\bsnm{Le~Cam},~\bfnm{Lucien}\binits{L.}}
(\byear{1960}).
\btitle{Locally asymptotically normal families of distributions}.
\bjournal{Univ. California Publ. Statist.}
\bvolume{3}
\bpages{37--98}.
\end{barticle}
\endbibitem

\bibitem[\protect\citeauthoryear{Le~Cam}{1986}]{LeCam}
\begin{bbook}[author]
\bauthor{\bsnm{Le~Cam},~\bfnm{Lucien}\binits{L.}}
(\byear{1986}).
\btitle{Asymptotic Methods in Statistical Decision Theory}.
\bseries{Springer Series in Statistics}.
\bpublisher{Springer New York}.
\end{bbook}
\endbibitem

\bibitem[\protect\citeauthoryear{Lindsay}{1989}]{Lindsay}
\begin{barticle}[author]
\bauthor{\bsnm{Lindsay},~\bfnm{B.~G.}\binits{B.~G.}}
(\byear{1989}).
\btitle{Moment matrices: applications in mixtures}.
\bjournal{The Annals of Statistics}
\bvolume{17}
\bpages{722--740}.
\end{barticle}
\endbibitem

\bibitem[\protect\citeauthoryear{Liu and Hancock}{2014}]{liu2014unrestricted}
\begin{barticle}[author]
\bauthor{\bsnm{Liu},~\bfnm{Min}\binits{M.}} \AND
  \bauthor{\bsnm{Hancock},~\bfnm{Gregory~R}\binits{G.~R.}}
(\byear{2014}).
\btitle{Unrestricted Mixture Models for Class Identification in Growth Mixture
  Modeling}.
\bjournal{Educational and Psychological Measurement}
\bvolume{74}
\bpages{557--584}.
\end{barticle}
\endbibitem

\bibitem[\protect\citeauthoryear{Martin}{2012}]{Ryan}
\begin{barticle}[author]
\bauthor{\bsnm{Martin},~\bfnm{Ryan}\binits{R.}}
(\byear{2012}).
\btitle{Convergence rate for predictive recursion estimation of finite
  mixtures}.
\bjournal{Statistics \& Probability Letters}
\bvolume{82}
\bpages{378--384}.
\end{barticle}
\endbibitem

\bibitem[\protect\citeauthoryear{Massart}{1990}]{Ma}
\begin{barticle}[author]
\bauthor{\bsnm{Massart},~\bfnm{P.}\binits{P.}}
(\byear{1990}).
\btitle{The tight constant in the Dvoretzky-Kiefer-Wolfowitz inequality}.
\bjournal{Annals of probability}
\bvolume{18}
\bpages{1269--1283}.
\end{barticle}
\endbibitem

\bibitem[\protect\citeauthoryear{McLachlan and Peel}{2000}]{McLP}
\begin{bbook}[author]
\bauthor{\bsnm{McLachlan},~\bfnm{G.}\binits{G.}} \AND
  \bauthor{\bsnm{Peel},~\bfnm{D.}\binits{D.}}
(\byear{2000}).
\btitle{Finite mixture models}.
\bseries{Wiley Series in Probability and Statistics: Applied Probability and
  Statistics}.
\bpublisher{Wiley-Interscience, New York}.
\end{bbook}
\endbibitem

\bibitem[\protect\citeauthoryear{Nguyen}{2013}]{Ngu}
\begin{barticle}[author]
\bauthor{\bsnm{Nguyen},~\bfnm{X.}\binits{X.}}
(\byear{2013}).
\btitle{Convergence of latent mixing measures in finite and infinite mixture
  models}.
\bjournal{Ann. Statist.}
\bvolume{41}
\bpages{370--400}.
\end{barticle}
\endbibitem

\bibitem[\protect\citeauthoryear{Pearson}{1894}]{Pear}
\begin{barticle}[author]
\bauthor{\bsnm{Pearson},~\bfnm{K.}\binits{K.}}
(\byear{1894}).
\btitle{Contributions to the theory of mathematical evolution}.
\bjournal{Philosophical Transactions of the Royal Society of London A}
\bvolume{185}
\bpages{71--110}.
\end{barticle}
\endbibitem

\bibitem[\protect\citeauthoryear{Rousseau and
  Mengersen}{2011}]{rousseau2011asymptotic}
\begin{barticle}[author]
\bauthor{\bsnm{Rousseau},~\bfnm{Judith}\binits{J.}} \AND
  \bauthor{\bsnm{Mengersen},~\bfnm{Kerrie}\binits{K.}}
(\byear{2011}).
\btitle{Asymptotic behaviour of the posterior distribution in overfitted
  mixture models}.
\bjournal{Journal of the Royal Statistical Society: Series B (Statistical
  Methodology)}
\bvolume{73}
\bpages{689--710}.
\end{barticle}
\endbibitem

\bibitem[\protect\citeauthoryear{Teh}{2010}]{Teh}
\begin{bincollection}[author]
\bauthor{\bsnm{Teh},~\bfnm{Yee~Whye}\binits{Y.~W.}}
(\byear{2010}).
\btitle{Dirichlet Process}.
In \bbooktitle{Encyclopedia of Machine Learning}
(\beditor{\bfnm{Claude}\binits{C.}~\bsnm{Sammut}} \AND
  \beditor{\bfnm{GeoffreyI.}\binits{G.}~\bsnm{Webb}}, eds.)
\bpages{280-287}.
\bpublisher{Springer US}.
\bdoi{10.1007/978-0-387-30164-8_219}
\end{bincollection}
\endbibitem

\bibitem[\protect\citeauthoryear{Titterington, Smith and Makov}{1985}]{Tit}
\begin{bbook}[author]
\bauthor{\bsnm{Titterington},~\bfnm{D.~M.}\binits{D.~M.}},
  \bauthor{\bsnm{Smith},~\bfnm{A.~F.~M.}\binits{A.~F.~M.}} \AND
  \bauthor{\bsnm{Makov},~\bfnm{U.~E.}\binits{U.~E.}}
(\byear{1985}).
\btitle{Statistical analysis of finite mixture distributions}.
\bseries{Wiley Series in Probability and Mathematical Statistics: Applied
  Probability and Statistics}.
\bpublisher{John Wiley \& Sons, Ltd., Chichester}.
\end{bbook}
\endbibitem

\bibitem[\protect\citeauthoryear{van~de Geer}{1996}]{VdG}
\begin{barticle}[author]
\bauthor{\bparticle{van~de} \bsnm{Geer},~\bfnm{Sara}\binits{S.}}
(\byear{1996}).
\btitle{Rates of convergence for the maximum likelihood estimator in mixture
  models}.
\bjournal{J. Nonparametr. Statist.}
\bvolume{6}
\bpages{293--310}.
\end{barticle}
\endbibitem

\bibitem[\protect\citeauthoryear{van~der Vaart}{1998}]{VDV}
\begin{bbook}[author]
\bauthor{\bparticle{van~der} \bsnm{Vaart},~\bfnm{A.~W.}\binits{A.~W.}}
(\byear{1998}).
\btitle{Asymptotic statistics}.
\bseries{Cambridge Series in Statistical and Probabilistic Mathematics}
\bvolume{3}.
\bpublisher{Cambridge University Press, Cambridge}.
\end{bbook}
\endbibitem

\bibitem[\protect\citeauthoryear{Yang}{2005}]{yang2005can}
\begin{barticle}[author]
\bauthor{\bsnm{Yang},~\bfnm{Yuhong}\binits{Y.}}
(\byear{2005}).
\btitle{Can the strengths of AIC and BIC be shared? A conflict between model
  indentification and regression estimation}.
\bjournal{Biometrika}
\bvolume{92}
\bpages{937--950}.
\end{barticle}
\endbibitem

\bibitem[\protect\citeauthoryear{Zhu and Zhang}{2004}]{Zhu2}
\begin{barticle}[author]
\bauthor{\bsnm{Zhu},~\bfnm{Hong-Tu}\binits{H.-T.}} \AND
  \bauthor{\bsnm{Zhang},~\bfnm{Heping}\binits{H.}}
(\byear{2004}).
\btitle{Hypothesis testing in mixture regression models}.
\bjournal{Journal of the Royal Statistical Society: Series B (Statistical
  Methodology)}
\bvolume{66}
\bpages{3--16}.
\end{barticle}
\endbibitem

\bibitem[\protect\citeauthoryear{Zhu and Zhang}{2006}]{Zhu}
\begin{barticle}[author]
\bauthor{\bsnm{Zhu},~\bfnm{Hongtu}\binits{H.}} \AND
  \bauthor{\bsnm{Zhang},~\bfnm{Heping}\binits{H.}}
(\byear{2006}).
\btitle{Asymptotics for estimation and testing procedures under loss of
  identifiability}.
\bjournal{Journal of Multivariate Analysis}
\bvolume{97}
\bpages{19--45}.
\end{barticle}
\endbibitem

\end{thebibliography}

\appendix

\section{Auxiliary matrix tool}

\begin{lem}
    \label{indep}
    Let $j$, $d_i$ and $d$ be positive integers such that
    $\sum_{i=1}^j d_i = d$. Consider numbers $\theta_1,\cdots,\theta_j$ all
    distinct. Write 
\[\mathcal{I} =\left\{ (i,\ell) \in \mathbb{N} : 1
      \le i \le j, 1 \le \ell \le d_i\right\}.\]
 Define for each $(i,\ell)\in\mathcal{I}$ a 
    $d$-dimensional column vector as follows:
\[ a_{i,\ell}[k]  =\dfrac{\theta _i^{k-\ell}}{(k-\ell)!}\1_{k\ge \ell},\quad 
1\le k\le d, 
\]
and stack these vectors in a $d\times d$ matrix 
\begin{equation}
  \label{A}
A(\theta_1,\ldots,\theta_j)=\left[ a_{1,1} | \dots |a_{1,d_1}|\dots |a_{j,1}| \dots | a_{j,d_j} \right]. 
\end{equation}
Then, the rank of $A(\theta_1,\ldots,\theta_j)$ is $d$.
\end{lem}
\begin{proof}
 Set for short $A= A(\theta_1,\ldots,\theta_j)$. Let
$\Lambda=(\lambda_{i,\ell})_{(i,\ell)\in\mathcal{I}}  $ be a vector such
that $A \Lambda  = 0$. Proving the lemma is equivalent to proving that
$\Lambda = 0$. Note that for each $k$
\begin{equation*}
  (A \Lambda)_k=\sum_{(i,\ell) \in \mathcal{I} } \lambda
  _{i,\ell}a_{i,\ell}[k]= \sum_{(i,\ell) \in \mathcal{I} }\lambda
  _{i,\ell}\dfrac{\theta _i^{k-\ell}}{(k-\ell)!}\1_{k\ge \ell}=0,
\end{equation*}
so that for any $(d-1)$-degree polynomial $P(x) = \sum_{k=0}^{d-1}
c_k\frac{x^k}{k!}$ , we have
\begin{equation}
  \label{eq:P}
 (c_0, \dots, c_{d-1}) A \Lambda = \sum_{k=0}^{d-1}c_k (A \Lambda)_{k+1}=\sum_{(i,\ell) \in \mathcal{I} } \lambda _{i,\ell} P^{(\ell-1)}(\theta _i)=0.
\end{equation} 
Set $P_i(x)=\prod_{\substack{k=1\\k\ne
    i}}^j (x-\theta_k)^{d_k}$ for each $i\in \lb 1,j\rb$.
Choosing successively  in \eqref{eq:P} the following polynomials
\begin{eqnarray*}
P(x) &=&(x-\theta_i)^{d_i-1}P_i(x),\\
P(x) &=& (x-\theta_i)^{d_i-2}P_i(x),\\
\vdots & & \vdots\\
P(x) &=&(x-\theta_i)^0P_i(x),
\end{eqnarray*}
yields successively $\lambda _{i,d_i}=0$, $\lambda _{i, d_i-1}=0$, ..., $\lambda
_{i, 1}=0$ and we are done.
\end{proof}

\begin{cor}
    \label{memenorme}
Let $\eps >0$ and define the set of $\eps$-separated vectors  in
$\Theta^j$ by 
\[\mathcal{D}_\eps=\left\{(\theta _i)_{1\le i\le j}: \forall i\ne i',\quad |\theta_i-\theta_{i'}|\ge \eps\right\}.\]
For any vector $\Lambda\in\R^d $ and any vector $(\theta _i)_{1\le i\le j}\in
\mathcal{D}_\eps$,
    \[
      \left\lVert A(\theta_1,\ldots,\theta_j) \Lambda  \right\rVert \underset{\eps}{\asymp}\left\lVert \Lambda  \right\rVert,
    \]
where $A(\theta_1,\ldots,\theta_j)$ is as in \eqref{A}.
\end{cor}
\begin{proof}
 Note that the norm $\|A(\theta_1,\ldots,\theta_j) \Lambda\|$
 is a continuous function of $((\theta_1,\ldots,\theta_j),\Lambda)$ on
 the compact space $\mathcal{D}_\eps\times S(0,1)$ where $S(0,1)$ is
 the  $d$-dimensional unit sphere.  Its infimum and supremum are attained on
 $\mathcal{D}_\eps\times S(0,1)$, say at $\left((\theta _{*i})_{1\le i
   \le j},\Lambda_*\right)$ and $\left((\theta _{i}^*)_{1\le i
   \le j},\Lambda^*\right)$ . Now, by Lemma \ref{indep},
$c_*(\eps)=\|A(\theta_{*1},\ldots,\theta_{*j}) \Lambda_*\|$  and $c^*(\eps)=\|A(\theta_{1}^*,\ldots,\theta_{j}^*) \Lambda^*\|$ are positive so that $ c_* \left\lVert
   \Lambda  \right\rVert \le  \left\lVert A(\theta_1,\ldots,\theta_j)
   \Lambda  \right\rVert\le c^* \left\lVert
   \Lambda  \right\rVert $
 for every $\Lambda$ and every $(\theta _i)_{1\le i\le j}$ in $\mathcal{D}_\eps$ .
\end{proof}

\section{Wasserstein distance and mixture on the tree $\T$}
\subsection{Key lemmas~\ref{scW} and \ref{lemrec}}
 Set  for any function $f$ on $\Theta$ and any $J\subset J_o$
\begin{equation}
  \label{eq:not}
  f(J)=\sum_{j\in J}\pi_{j,n} f\left(\theta_{j,n}\right).
\end{equation}
In particular for $f(\cdot)=F(x,\cdot)$, we have 
\[F(x,J)=\sum_{j\in J}\pi_{j,n}F(x,\theta_{j,n}),\]
so that $F(x,\Delta G_n)=F(x,J_o)$.  
 Set also for short
 \begin{equation}
   \label{eq:piJ}
 \pi(J)=\sum_{j\in J}\pi_j. 
 \end{equation}

\underline{Proof of Lemma~\ref{scW}}. With the above notations, we have to show \eqref{scaleW}.


In what follows $n$ is fixed
  and thus skipped in the $\theta_j$'s, $\pi_j$'s and $\eps_s$'s. For each distinct $J$, we pick an arbitrary
  $j\in J$ and set $\theta_J=\theta_j$.  Let $f$ be $1$-Lipschitz on $\Theta$. We first prove by recurrence
  that for any node $J$ of the tree,
\begin{equation}
\label{eq:recfj}
 f(J)\lec \pi(J)f(\theta_J)+\max_{I\in\mathrm{Desc}(J)}\eps_{\fs_\pi(I)}\eps_{\fs(\parent{I})}.
\end{equation}
If $J$ has $\fs$-diameter zero, then $f(J)=\pi(J)f(\theta_J)$ and
\eqref{eq:recfj} is satisfied. Next, if $J$ has children $J_1$ that satisfy \eqref{eq:recfj}, we compute 
\begin{eqnarray*}
   f(J) &=& \sum_{J_1\in \C{J}}f(J_1)\\
   &\lec & \sum_{J_1\in \C{J}}\left[\pi(J_1)f(\theta_{J_1})+\max_{I\in \D{J_1}}\eps_{\fs_\pi(I)}\eps_{\fs(\parent{I})}\right]\\
&\le & \pi(J)f(\theta_J) +\!\!\!\!\sum_{J_1\in
  \C{J}}\!\!\!|\pi(J_1)|\underbrace{|f(\theta_{J_1})-f(\theta_J)|}_{\le
    |\theta_{J_1}-\theta_J|} +\!\!\!\max_{I\in
    \D{J_1}}\eps_{\fs_\pi(I)}\eps_{\fs(\parent{I})}.
\end{eqnarray*}
Since $|\pi(J_1)|$ is of order $\eps_{\fs_\pi(J_1)} $ and
$|\theta_{J_1}-\theta_J|$ is of order $\eps_{\fs(J_1)}$ we see that \eqref{eq:recfj} holds for $J$ and in particular for $J_o$
where $\pi(J_o)=0$.

To prove the reverse inequality, let $J\subsetneq J_o$ such that $\eps_{\fs_\pi(J)}\eps_{\fs(\parent{J})}$ is
maximal. Set 
\[\fe(J)=\min_{j\notin J}|\theta_{j}-\theta_J|\quad\text{ and } \quad\fd(J)=\max_{j\in J}|\theta_{j}-\theta_J|\]
so that $\fe(J)\ge \fd(J)$. Consider the following $1$-Lipschitz
function $f$ on $\Theta$ 
\[f(\theta)=-\mathrm{sgn}(\pi(J))\times\min\{\fe(J)-\fd(J),[|\theta-\theta_J|-\fd(J)]_+\}\] 
so that 
\[f(J)=0 \quad\text{ and } \quad f(J_o) = f(J_o\setminus J)= |\pi(J)|[\fe(J)-\fd(J)].\]
Since  $|\pi(J)|$ is of order $\eps_{\fs_\pi(J)}$ and $\fe(J)$
is at least of order $\eps_{\fs(\parent{J})}$ and $\fd(J)$
is of order $\eps_{\fs(J)}$, we deduce
\begin{equation*}
 f(J_o)\gec\max_{J\in
    \D{J_o}}\eps_{\fs_\pi(J)}\eps_{\fs(\parent{J})}.
\end{equation*}  
It remains to note that  $W\left(\Delta G_n\right)=\sup_{\Vert f\Vert_{\text{Lip}}\le 1}f(J_o)$.
\vspace{1em}

\underline{Proof of Lemma~\ref{lemrec}}. We shall use  Taylor expansions along the tree $\T$ to express the order of $F(x,\Delta G_n)$ in terms of the scaling functions $\eps_s(n)$.
Recall Assumption B($k$) in the main paper: the densities family $\{f(\cdot,\theta),\theta\in\Theta\}$ satisfies, with $F(x,\theta)=\int_{-\infty}^xf(\cdot,\theta)\dd\lambda$,
\begin{itemize}
\item $\{F(\cdot,\theta),\theta\in\Theta\}$ is $k$-strongly identifiable,
\item For all $x$, $F(x,\theta)$ is $k$-differentiable w.r.t. $\theta$,
\item There is a uniform continuity modulus $\omega(\cdot)$ such that 
\[\sup_x\big|F^{(k)}(x,\theta_2)-F^{(k)}(x,\theta_1)\big|\le \omega(\theta_2-\theta_1)\]
with $\lim_{h\to 0}\omega(h)=0$.
\end{itemize}

Recall notations \eqref{eq:not} and \eqref{eq:piJ}. If $J$ is an end of the
tree $\T$, then it satisfies
$\fs(J)=0$, all the $\theta_j$ for $j\in J$ are equal, and
$F(x,J)=\pi(J)F(x,\theta_J)$. In this case, the choices $a_{J}(k)=\pi(J)\1_{\{k=0\}}$ and $R(x,J)=0$ work. 

\emph{Assume now that Lemma~\ref{lemrec} holds for any node $I$ with
  parent $J=\parent{I}$ in the tree $\T$.}   We want to pass the
estimates of $I$ to the parent $J$. By assumption on $I$,
    \begin{equation}
      \label{eq:FxI}
       F(x,I) -R(x,I) =  \sum_{\ell=0}^{2m} a_{I}(\ell) \eps_{\fs(I)}^\ell  F^{(\ell)}(x, \theta _I). 
    \end{equation}
    Assuming without loss of generality that $\theta_J\le \theta_I$, we apply
 Taylor's formula with remainder to $F^{(\ell)}(x, \theta _I)$ at $\theta_J$ and obtain
 \[F^{(\ell)}(x, \theta_I)-\!\!\sum_{k=\ell}^{2m-1}\frac{(\theta_I-\theta_J)^{k-\ell}}{(k-\ell)!}
F^{(k)}(x,\theta_J) = \!\! \int_{\theta_J}^{\theta_I}\!\! \frac{(\theta_I-\xi)^{2m-1-\ell}}{(2m-1-\ell)!}F^{(2m)}(x, \xi)d\xi . \] 
So that using Assumption~B(2m),
\begin{align*}
F^{(\ell)}(x, \theta_I)-&\sum_{k=\ell}^{2m}\frac{(\theta_I-\theta_J)^{k-\ell}}{(k-\ell)!}
F^{(k)}(x,\theta_J)  \\
&= \int_{\theta_J}^{\theta_I} \frac{(\theta_I-\xi)^{2m-1-\ell}}{(2m-1-\ell)!}\left[F^{(2m)}(x, \xi)-F^{(2m)}(x,\theta_J)\right]d\xi\\ 
&=\frac{(\theta_I-\theta_J)^{2m-\ell}}{(2m-1-\ell)!}\, O\left(  \sup_{\xi\in
    [\theta_J,\theta_I]}|F^{(2m)}(x,\xi)-F^{(2m)}(x,\theta_J)|\right)\\
&=o\left((\theta_I-\theta_J)^{2m-\ell}\right),
\end{align*}
and by setting
\begin{equation}
  \label{etalIhI}
  \vartheta_I   =     \frac{\theta_I - \theta
    _J}{\eps_{\fs(J)}}\quad\text{and}\quad a'_I(\ell) =  a_I(\ell) \left( \frac{\eps_{\fs(I)}}{\eps_{\fs(J)}} \right)^\ell, 
\end{equation}
we obtain
\begin{equation*}
F^{(\ell)}(x, \theta_I)=  \sum_{k=\ell}^{2m}\eps_{\fs(J)}^k\frac{\vartheta_I^{k-\ell}}{(k-\ell)!}F^{(k)}(x,\theta_J)+\eps_{\fs(J)}^{2m-\ell}o(1),
\end{equation*}
and substituting in \eqref{eq:FxI} and changing the order of  summation, we get
\begin{align*}
  F(x,I) -R(x,I)
 = \sum_{k=0}^{2m} \eps_{\fs(J)}^k F^{(k)}(x, \theta_J)\sum_{\ell=0}^{k}&
  a'_I(\ell)\frac{\vartheta_I^{k-\ell}}{(k-\ell)!}\\
& + \eps_{\fs(J)}^{2m}\, \max_{\ell\le 2m}|a'_I(\ell) | \,o(1).\\
\end{align*}
Adding up over the children $I$ of $J$,  we obtain 
\begin{equation*}
  F(x,J)=\sum_{k=0}^{2m}a_J(k)\eps_{\fs(J)}^k F^{(k)}(x, \theta_J)+R(x,J) ,
\end{equation*}
with 
\begin{eqnarray}
\label{etakJ}
  a_J(k)&=&\sum_{I\in \mathrm{Child}(J)}\sum_{\ell=0}^{k}
  a'_I(\ell)\frac{\vartheta_I^{k-\ell}}{(k-\ell)!},\\
\label{RxJ}
  R(x,J)&=&\sum_{I\in \mathrm{Child}(J)}\left[\eps_{\fs(J)}^{2m}\, \max_{\ell\le 2m}|a'_I(\ell) | \,o(1) + R(x,I)\right].
\end{eqnarray}
 
\emph{Proof of (\ref{coeffsbornes}) for the node $J=\parent{I}$.} From
 \eqref{etakJ} for $k=0$ and \eqref{etalIhI} and recurrence hypothesis on $I$, we have 
\[a_J(0)=\sum_{I\in \mathrm{Child}(J)}a'_{I}(0)=\sum_{I\in \mathrm{Child}(J)}a_I(0)=\sum_{I\in \mathrm{Child}(J)}\sum_{j\in I}\pi_j=\sum_{j\in J}\pi_j.\]
Moreover, since $|\vartheta_I|\lec 1$
for each child $I$ of $J$, Equation \eqref{etakJ} yields 
\[|a_J(k)|\lec \max_{\substack{\ell\le k \\ I\in\mathrm{Child}(J)}} |a'_{I}(\ell)|.\]
 Furthermore, from \eqref{etalIhI} we have $|a'_{I}(\ell)| \le |a_{I}(\ell)|$ since $\eps_{\fs(I)}\le
 \eps_{\fs(J)}$. And by assumption on $I$, we have $|a_I(\ell)|\lec 1$ so that
 $|a'_I(\ell)|$ and thus $|a_{J}(k)|$ are  also of order one and (\ref{coeffsbornes}) is established. 

 \emph{We turn to the proof of  (\ref{premierscoeffs}) for $J=\parent{I}$.} The first step is to show that
\begin{equation}
  \label{eq:applilm}
  \max_{k<d_J}|a_J(k)|\asymp \max_{\substack{\ell< d_I \\ I\in \mathrm{Child}(J)}} |a'_I(\ell)|.
\end{equation}
>From \eqref{etakJ}, write $a_J(k)=a^{(1)}_J(k)+a^{(2)}_J(k)$ with 
\begin{eqnarray}
\label{a1} a^{(1)}_J(k) &=&\sum_{I\in \C{J}}\sum_{\ell=0}^{d_I-1}a'_I(\ell)\frac{\vartheta_I^{k-\ell}}{(k-\ell)!}\1_{k\ge \ell} ,\\
\label{a2} a^{(2)}_J(k) &= & \sum_{I\in \C{J}}\sum_{\ell=d_I}^{k}a'_I(\ell)\frac{\vartheta_I^{k-\ell}}{(k-\ell)!}\1_{k\ge \ell}.
\end{eqnarray}
For any two distinct children $I$ and $I'$ of $J$, \eqref{etalIhI} gives
\begin{equation}
\label{sep}
|\vartheta_{I} - \vartheta_{I'}| = \eps_{\fs(J)}^{-1} |\theta_{I} - \theta _{I'}|\asymp 1,
\end{equation}
so that $\{\vartheta_I\}_{I\in\C{J}}$ is $\eps$-separated for some
$\eps>0$. Hence, by Corollary~\ref{memenorme}, if we set $\Lambda=\left(a'_I(\ell)\right)_{0\le \ell\le d_I-1}$, we get 
\[\max_{k<d_J}\left|a^{(1)}_J(k)\right|\asymp \max_{\substack{\ell< d_I \\ I\in \C{J}}} |a'_I(\ell)|.\]
Now, to obtain \eqref{eq:applilm}, we see from \eqref{etakJ} that it's enough to show 
\begin{equation}
  \label{eq:negl}
\max_{k<d_J}\left|a^{(2)}_J(k)\right|=o\left(\max_{\ell< d_I} |a'_I(\ell)|\right). 
\end{equation}
Since $|\vartheta_I|\lec 1$, we have from \eqref{a2}
\begin{equation}
  \label{eq:gro}
 \left|a^{(2)}_J(k)\right|\lec \max_{d_I\le \ell\le k} |a'_I(\ell)|.
\end{equation}
By assumption on $I$, we also have $\|a_I\|\asymp \max_{\ell< d_I} |a_I(\ell)|$, so that
\begin{eqnarray*}
 \frac{\eps_{\fs(I)}}{\eps_{\fs(J)}}\cdot \max_{\ell< d_I}
  |a'_I(\ell)| =\max_{\ell< d_I}
  |a_I(\ell)|\left(\frac{\eps_{\fs(I)}}{\eps_{\fs(J)}}\right)^{\ell +1}& \gec & \|a_I\|\left(\frac{\eps_{\fs(I)}}{\eps_{\fs(J)}}\right)^{d_I} \\
&\ge & \max_{d_I\le \ell\le k} |a'_I(\ell)| ,
\end{eqnarray*}
where the last inequality comes from \eqref{etalIhI}. Thus,  
\begin{equation}
  \label{eq:pto}
 \max_{d_I\le \ell\le k} |a'_I(\ell)|= o\left(\max_{\ell< d_I} |a'_I(\ell)|\right) ,
\end{equation}
so that \eqref{eq:gro} and \eqref{eq:pto} yield \eqref{eq:negl} and \eqref{eq:applilm} is proved. 

The second step is to prove 
\begin{equation}
  \label{equi}
 \|a_J\|\asymp \max_{k< d_J} |a_J(k)|. 
\end{equation}
The non-trivial part is $\|a_J\|\lec \max_{k< d_J}
|a_J(k)|$;  it is equivalent to show 
\[\max_{k\ge d_J} |a_J(k)|\lec  \max_{k< d_J} |a_J(k)|.\]
 By the definition \eqref{etakJ} of $a_J(k)$,
\eqref{eq:pto} and \eqref{eq:applilm}, we have 
\begin{eqnarray*}
\max_{k\ge d_J} |a_J(k)|  &\lec  &  \max_{k\ge d_J}\sum_{I\in
                                      \C{J}}\max_{\ell\le k}
                                      |a'_I(\ell)|\\
& \lec  &  \sum_{I\in \mathrm{Child}(J)}\max_{\ell<d_I} |a'_I(\ell)| \lec   \max_{\substack{\ell< d_I \\ I\in \mathrm{Child}(J)}} |a'_I(\ell)| \lec  \max_{k< d_J} |a_J(k)|.
\end{eqnarray*}
The proof of (\ref{premierscoeffs}) is complete.

\emph{We turn to the proof of (\ref{borneinfcoeff}) for $J=\parent{I}$.} From \eqref{equi},
\eqref{eq:applilm} and \eqref{etalIhI}, we get 
\begin{eqnarray}
  \|a_J\| \gec \max_{k< d_J} |a_J(k)| \gec \max_{\substack{\ell< d_I \\ I\in \mathrm{Child}(J)}}|a'_I(\ell)|
&\gec &\max_{\substack{\ell< d_I \\ I\in \mathrm{Child}(J)}}|a_I(\ell)|\left(\frac{\eps_{\fs(I)}}{\eps_{\fs(J)}}\right)^\ell\nonumber\\
&\gec &\max_{I\in\mathrm{Child}(J)}\|a_I\|\left(\frac{\eps_{\fs(I)}}{\eps_{\fs(J)}}\right)^{d_J-1}.\label{mino}
\end{eqnarray}
Moreover (\ref{borneinfcoeff}) for $I\in\C{J}$ gives since $d_I\le d_J$
\begin{eqnarray*}
  \|a_I\|\left(\frac{\eps_{\fs(I)}}{\eps_{\fs(J)}}\right)^{d_J-1} &\gec&
\max_{I'\in\D{I}}\eps_{\fs_\pi(I')}\left(\frac{\eps_{\fs(\parent{I'})}}{\eps_{\fs(I)}}\right)^{d_I-1}\left(\frac{\eps_{\fs(I)}}{\eps_{\fs(J)}}\right)^{d_J-1}\\
&\gec& \max_{I'\in\D{I}}\eps_{\fs_\pi(I')}\left(\frac{\eps_{\fs(\parent{I'})}}{\eps_{\fs(J)}}\right)^{d_J-1}.\\
\end{eqnarray*}
In addition, (\ref{coeffsbornes}) implies $\|a_J\|\gec |a_J(0)|=|\pi(J)|\asymp
\eps_{\fs_\pi(J)}$ and similarly, from \eqref{eq:applilm}, \eqref{equi} and (\ref{coeffsbornes}) for
$I$, $\|a_J\|\gec |a'_I(0)|=|a_I(0)|=|\pi(I)|\asymp
\eps_{\fs_\pi(I)}$ so that (\ref{borneinfcoeff}) is established for $J$.

\emph{We finally prove (\ref{remaind}) for $J=\parent{I}$.} From \eqref{RxJ}, \eqref{eq:pto}, assumption (\ref{remaind}) for $I$ and \eqref{eq:applilm}, we have 
\begin{eqnarray*}
  R(x,J) &\lec & \max_{I\in\mathrm{Child}(J)}\left[\eps_{\fs(J)}^{2m}\, \max_{\ell\le
  2m}|a'_I(\ell) | \,o(1) + R(x,I)\right]\\
 &\lec & \max_{I\in\mathrm{Child}(J)}\left[\eps_{\fs(J)}^{2m}\max_{\ell<d_I}|a'_I(\ell) | \,o(1) + o\left(\|a_I\|\eps_{\fs(I)}^{2m}\right)\right]\\
 &\lec &\eps_{\fs(J)}^{2m}\|a_J\| \,o(1)+\max_{I\in\mathrm{Child}(J)}o\left(\|a_I\|\eps_{\fs(I)}^{2m}\right),
\end{eqnarray*}
and in addition, for each child $I$ of $J$, from \eqref{mino},
\begin{equation*}
  \|a_I\|\eps_{\fs(I)}^{2m}  = \|a_I\|\eps_{\fs(I)}^{d_I-1}\eps_{\fs(I)}^{2m+1-d_I}\lec  \|a_J\|\eps_{\fs(J)}^{d_I-1}\eps_{\fs(I)}^{2m+1-d_I}\le  \|a_J\|\eps_{\fs(J)}^{2m},
\end{equation*}
and we are done.

\emph{The proof of (\ref{epsep})} is already established in \eqref{sep}.

\def\x{\ref{sec:introduction}}
\section{From local to global: Theorem~\MakeLowercase{\ref{orders}.\ref{local}} implies
Theorem~\MakeLowercase{\ref{orders}.\ref{global}}} 
\label{closeinspec} Set 
\[L=\inf_{\substack{G_1,G_2\in\Glm\\G_1\ne G_2}} \frac{\left\lVert F(\cdot, G_1) - F(\cdot, G_2) \right\rVert _\infty}{W(G_1, G_2)^{2m-1}}\]
and consider a sequence
  $(G_{1,n},G_{2,n})$ in $\Glm^2$ with $G_{1,n}\ne G_{2,n} $ for
  each $n$ and such that 
  \begin{equation}
\label{mi}
   \frac{\left\lVert F(\cdot, G_{1,n}) - F(\cdot, G_{2,n}) \right\rVert
     _\infty}{W(G_{1,n}, G_{2,n})^{2m-1}}\xrightarrow[n\to \infty]{}L .
  \end{equation}
We can assume that 
$(G_{1,n},G_{2,n})$ converges to some limit $(G_{1,\infty},G_{2,\infty})$ in the compact set
$\Glm^2$. Set  $w=W(G_{1,\infty}, G_{2,\infty})$, $\Delta G_n=G_{1,n}-
G_{2,n}$ and distinguish  two cases : $w>0$ and $w=0$. 

If $w>0$, by identifiability, there is a $x_0$ such that
$\delta_0=|F(x_0,G_{1,\infty})-F(x_0,G_{2,\infty})|>0$. Then, for all $n$
\begin{equation}
  \label{eq:1}
 \frac{\left\| F(\cdot, \Delta G_n) \right\|_\infty}{W(\Delta
   G_n)^{2m-1}}\ge \frac{\left| F(x_0, \Delta G_n) \right|}{W(\Delta G_n)^{2m-1}}.
\end{equation}
By assumption, $W(\Delta G_n)$ tends to $w$. Moreover, the numerator
of the r.h.s. of \eqref{eq:1} tends to $\delta_0$ since the function $\theta\mapsto
F(x_0, \theta)$ is $K_{x_0}$-Lipschitz  with $K_{x_0}=\max_{\theta\in
  \Theta}|F^{(1)}(x_0,\theta)|$.  As a
consequence,  \eqref{eq:1} and \eqref{mi} give
Theorem~\ref{orders}.\ref{global}.
 by choosing $\delta=\delta_0/w^{2m-1}$.
  
If  now $w=0$, set $G_0= G_{1,\infty}$
which is in $\mathcal{G}_{m_0}$ with some  $m_0$ at most $m$. Consider $\eps
>0$ and $\delta>0$ as defined in Theorem~\ref{orders}.\ref{local}. ; for $n$ large enough,
say $n\ge n_0$, $W(G_{i,n},G_0)$, $i=1,2$, are less than $\eps$ so that 
\[\inf_{n\ge n_0}\frac{\left\lVert F(\cdot,\Delta G_n) \right\rVert
  _{\infty}}{W(\Delta G_n)^{2m -2m_0+1}}>\delta.\]
Moreover, for $n$ large enough, say $n \ge n_1$, $ W(\Delta G_n)$
is smaller than one  so that 
$W(\Delta G_n)^{2m -2m_0+1}$ is more than $ W(\Delta G_n)^{2m-1}$
and thus for all $n\ge n_0+n_1$,
\begin{equation*}
 \frac{\left\lVert F(\cdot, \Delta G_n)
   \right\rVert _{\infty}}{W(\Delta G_n)^{2m-1}}\ge  \inf_{n\ge
   n_0+n_1}\frac{\left\lVert F(\cdot, \Delta G_n)\right\rVert
   _{\infty}}{W(\Delta G_n)^{2m -2m_0+1}}> \delta
\end{equation*}
which gives $L\ge \delta$ in the limit and Theorem~\ref{orders}.\ref{global}. in that case.

The proof of Theorem~\ref{orders} is complete.

\section{$(p,q)$-smoothness}
\subsection{Inherited smoothness for mixing distributions}
 Being $(p,q)$-smooth ensures finiteness of similar integrals when some $\theta_j$ are replaced with mixing distributions with components close to the $\theta_j$:
    \begin{prop}
        \label{tversmix}
Assume that the family $\{f(\cdot,\theta),\theta\in\Theta\}$ is $(p, q)$-smooth and let $\eps >0$ as in Definition~\ref{Eia}.\ref{proche}. Let also $\pi_0>0$, $\theta_0\in\Theta$ and positive integers $m, m_0$ with $m\ge m_0$. Define
        mixing distributions
 \[G_n = \sum_{j=1}^m \pi_{j,n} \delta _{\theta _{j,n}}\] 
such that 
\begin{itemize}
\item For all $ j \in\lb m_0,m\rb$, $\theta _{j,n}\xrightarrow[n\to\infty]{} \theta _0$, 
\item For all $n$ large enough, $\sum_{j=m_0}^m \pi_{j,n} \ge  \pi_0$. 
\end{itemize}
Then for any $\theta'$ satisfying  $\left|\theta' - \theta _0 \right| <
\eps / 2$, for any mixing distribution $G$:
        \begin{align}
\label{ub}
            \mathbb{E}_{G}\left|\frac{f^{(p)}(\cdot, \theta')}{f(\cdot,G_n)} \right|^q & \underset{\theta_0,\pi_0}{\lec} 1
        \end{align}
for $n$ large enough. If, in addition, the function $x\mapsto\left|f^{(p)}(x, \theta _0)\right|$ has
nonzero integral under $\lambda$, then  for any mixing distribution $G$,
\begin{align}
\label{lb}
    \mathbb{E}_{G}\left|\frac{f^{(p)}(\cdot, \theta _0)}{f(\cdot,G)} \right|^q & \underset{\theta_0}{\gec} 1.
\end{align}
    \end{prop}

    \begin{proof}
      For large $n$ and all $ j\in\lb m_0,m\rb$, we have
      $\left|\theta _{j,n} - \theta'\right|< \eps$ for all $\theta'$
      such that $\left|\theta' - \theta _0 \right| < \eps / 2$. For
      all such $(j,n)$ and all $\theta $, by $(p,q)$-smoothness and compactness and continuity, there is a finite $C$ such that  
\[\mathbb{E}_{\theta}\left|\frac{f^{(p)}(\cdot,
    \theta')}{f(\cdot, \theta _{j,n})} \right|^q \le C.\]
Since $f(x,G)$ is a convex
 combination of some $f(x,\theta)$, we may replace $\E_{\theta}$ by
 $\E_G$ in the former expression. Since the function $1/y^q$ is
 convex on positive reals, by Jensen inequality, setting
 $A=\sum_{j=m_0}^m \pi_{j,n}$, 
 \[ \sum_{j=m_0}^m \frac{\pi_{j,n}}{A} \left| \frac{f^{(p)}(x, \theta')}{f(x,
     \theta _{j,n})} \right|^q\ge\left|
   \frac{f^{(p)}(x, \theta')}{\sum_{j=m_0}^m \frac{\pi_{j,n}}{A}f(x, \theta _{j,n})}\right|^q \ge  A^q\left|
   \frac{f^{(p)}(x, \theta')}{f(x, G_{n})}\right|^q,\]
and taking expectations with respect to $G$ we obtain the upper bound:
\[\E_G\left|
   \frac{f^{(p)}(\cdot, \theta')}{f(\cdot, G_{n})}\right|^q\le \frac{C}{A^q}\le \frac{C}{\pi_0^q}.\]

The lower bound does not depend on $(p, q )$-smoothness. It is a simple consequence of rewriting:
        \begin{align*}
            \mathbb{E}_{G}\left|\frac{ f^{(p)}(\cdot, \theta _0)}{f(\cdot,G)} \right|^q & =  \int\left|\frac{ f^{(p)}(x, \theta _0)^q}{f(x,G)^{q-1}} \right|   \dd\lambda(x).     
        \end{align*}
By assumption, there is a
        set $B$ of measure $\lambda(B)=M>0$ on which the function $f^{(p)}(x, \theta
        _0)$ is more than some $\delta >0$. Now, for $M$ small enough,
 the set $B\cap
        \{f(x,G)\le 2/M\}$ is of measure at least $M/2$ and thus 
\[\int\left|\frac{ f^{(p)}(x, \theta _0)^q}{f(x,G)^{q-1}} \right|
\dd\lambda(x)     \ge \left[\frac{M}{2}\right]^{q+1} \delta^q.\]
\end{proof}

\subsection{ \MakeLowercase{$(p,q)$}-smoothness of exponential families}
\label{exp_smooth}

Given our definition of $(p,q)$-smoothness, it only makes sense to consider one-parameter one-dimensional families. However, generalisation to higher dimensions should be easy.

Let us consider an exponential family with natural parameter $\theta\in \Theta _0 \subset \R$, so that 
\[f(x, \theta ) = h(x) g(\theta) \exp(\theta T(x)),\]
with $g \in C^{\infty}$ and a sufficient one-dimensional statistic $T(x)$. Consider $\Theta $ such that its $\varepsilon $-neighbourhood $\Theta\oplus B(0, \varepsilon )$ is included in $\Theta_0 $. Then $\left\{ f(\cdot, \theta), \theta \in \Theta  \right\} $ is $(p,q)$-smooth for any $p$ and $q$. Indeed,
        \begin{eqnarray*}
          f^{(p)}(x, \theta') & = &
                                      h(x)\e^{\theta'T(x)}\left[\sum_{k=0}^p\binom{p}{k}g^{(k)}(\theta')T^{p-k}(x)\right]\\
\frac{ f^{(p)}(x, \theta')}{ f(x, \theta'')}& = &
                                      \frac{\e^{(\theta'-\theta'')T(x)}}{g(\theta'')}\left[\sum_{k=0}^p\binom{p}{k}g^{(k)}(\theta')T^{p-k}(x)\right]\\
\left|\frac{ f^{(p)}(x, \theta')}{ f(x, \theta'')}\right|^q& = &
                                     \frac{\e^{q(\theta'-\theta'')T(x)}}{g^q(\theta'')}\left|\sum_{k=0}^p\binom{p}{k}g^{(k)}(\theta')T^{p-k}(x)\right|^q
\end{eqnarray*}
so that we get from \eqref{Epq}
\[E_{p,q}(\theta, \theta', \theta'')  = \frac{g(\theta) \E_{\theta+q(\theta'-\theta'')}\left|\sum_{k=0}^p\binom{p}{k}g^{(k)}(\theta')T^{p-k}(\cdot)\right|^q}{g^q(\theta'')g(\theta+q(\theta'-\theta''))}.\]
        
        Since all the moments of the sufficient statistic $T(x)$ are finite under a distribution in the exponential family, and since $\theta + q\theta' - q \theta''$ is in $\Theta _0$ for $ (\theta' -  \theta'') < \eps/q$, we obtain the finiteness of $E_{p,q } (\theta, \theta', \theta'') $. Continuity is clear.

\section{Jacobian calculus}
\label{jaco}
The map 
\[\phi :(\pi_1,\ldots,\pi_d,\theta_1,\ldots,\theta_d)\mapsto 
\left(\sum_1^d\pi_j,\sum_1^d\pi_j\theta_j,\sum_1^d\pi_j\theta_j^2,\ldots,\sum_1^d\pi_j\theta_j^{2d-1}\right)\]
defined on $\R^{2d}$ has the following Jacobian : 
\[
J(\phi)=(-1)^{\frac{(d-1)d}{2}}\,\pi_1\cdots\pi_d \prod_{1\le j<k\le d}(\theta_j-\theta_k)^4.  \]
To prove this, note that
\begin{equation*}
 J(\phi) = 
\begin{vmatrix}
 1 & \cdots & 1 & 0 & \cdots & 0\\
\theta_1 &  \cdots & \theta_d & \pi_1& \cdots & \pi_d\\
\theta_1^2 &  \cdots & \theta_d^2 & 2\pi_1\theta_1& \cdots & 2\pi_d\theta_d\\
\vdots &            & \vdots & \vdots &        & \vdots \\ 
\theta_1^{2d-1} &  \cdots & \theta_d^{2d-1} & (2d-1)\pi_1\theta_1^{2d-2}& \cdots & (2d-1)\pi_d\theta_d^{2d-2}
\end{vmatrix}
\end{equation*}
so that $J(\phi) =\pi_1\cdots\pi_d \,\Delta_d $ with 
\begin{equation*}
\Delta_d = \begin{vmatrix}
 1 & \cdots & 1 & 0 & \cdots & 0\\
\theta_1 &  \cdots & \theta_d & 1& \cdots & 1\\
\theta_1^2 &  \cdots & \theta_d^2 & 2\theta_1& \cdots & 2\theta_d\\
\vdots &            & \vdots & \vdots &        & \vdots \\ 
\theta_1^{2d-1} &  \cdots & \theta_d^{2d-1} & (2d-1)\theta_1^{2d-2}& \cdots & (2d-1)\theta_d^{2d-2}
\end{vmatrix}.
\end{equation*}
Note that, if $P$ is any  polynomial of degree $2d-1$ with leading
coefficient one, the last row of $\Delta_d$ can be replaced by 
\[ [ P(\theta_1)   \cdots  P(\theta_d)\  P'(\theta_1) \cdots  P'(\theta_d)],\]
and choosing 
$P(\theta)=(\theta-\theta_d)\prod_{1\le j\le d-1}(\theta-\theta_j)^2$, we get 
\begin{equation*}\Delta_d=P'(\theta_d) \begin{vmatrix}
 1 & \cdots & 1 & 0 & \cdots & 0\\
\theta_1 &  \cdots & \theta_d & 1& \cdots & 1\\
\theta_1^2 &  \cdots & \theta_d^2 & 2\theta_1& \cdots & 2\theta_{d-1}\\
\vdots &            & \vdots & \vdots &        & \vdots \\ 
\theta_1^{2d-2} &  \cdots & \theta_d^{2d-2} & (2d-2)\theta_1^{2d-3}& \cdots & (2d-2)\theta_{d-1}^{2d-3}
\end{vmatrix}.
\end{equation*}
Again, if $Q$ is any  polynomial of degree $2d-2$ with leading
coefficient one, the last row  can be replaced by 
\[ [ Q(\theta_1)   \cdots  Q(\theta_d)\  Q'(\theta_1) \cdots  Q'(\theta_{d-1})],\]
and choosing  $\displaystyle Q(\theta)=\prod_{1\le j\le
  d-1}(\theta-\theta_j)^2$, we obtain the recurrence formula
\begin{equation*}
\Delta_d=(-1)^{d-1} P'(\theta_d) Q(\theta_d)\Delta_{d-1}=(-1)^{d-1} \prod_{j=1}^{d-1}(\theta_d-\theta_j)^4 \Delta_{d-1}.
\end{equation*} 
By iteration, we get 
\begin{eqnarray*}
 \Delta_d 
&=&
(-1)^{d-1+d-2+\cdots+1}\prod_{k=2}^d\prod_{j=1}^{k-1}(\theta_k-\theta_j)^4
\Delta_1 \\
&=&
(-1)^{\frac{(d-1)d}{2}}\prod_{1\le j<k\le d}(\theta_k-\theta_j)^4
\end{eqnarray*}
since $\Delta_1=1$. The proof is complete.

\section{Taylor expansions and $L^p$-convergences in the proof of Theorem~\ref{LAN}}
\subsection*{The $Y_{i,n} (u)$'s and $Z_{i,n}$'s are centered}
 Recall the definition \eqref{mixdist} of $G_n(u)$ and set for short 
\[\theta_{j,n}=\theta_{m_0}+n^{-1/(4d-2)}h_j(u)\]
with $d=m-m_0+1$. By definition of the mixtures, we have
\begin{equation*}
f\left(x,G_n(u)\right)-f\left(x,G_0\right)=\pi_{m_0}\sum_{j=m_0}^m\pi_{j,n}(u)\left[f(x,\theta_{j,n}(u))-f(x,\theta_{m_0})\right],
\end{equation*}
 and by Taylor expansion with remainder, 
 \begin{align*}
  f(x,\theta_{j,n}(u))-f(x,\theta_{m_0})=\sum_{k=1}^{2d-1}&\left(\frac{h_j(u)}{n^{1/(4d-2)}}\right)^kf^{(k)}(x,\theta_{m_0})\\
 &+\int_{\theta_{m_0}}^{\theta_{j,n}(u)}f^{(2d)}(x,\theta)\frac{(\theta_{j,n}(u)-\theta)^{2d-1}}{(2d-1)!}\dd\theta,
 \end{align*}
 so that  we get by linearity
 \begin{align}  \label{FxGnu}\frac{f(x,G_n(u))-f(x,G_0)}{\pi_{m_0}}=\sum_{k=1}^{2d-1}\frac{\mu_k}{n^{k/(4d-2)}}f^{(k)}(x,\theta_{m_0})+R_n(x,u)
 \end{align}
 with
 \begin{equation*}
  R_n(x,u)=\sum_{j=m_0}^m\pi_{j,n}(u) \int_{\theta_{m_0}}^{\theta_{j,n}(u)}f^{(2d)}(x,\theta)\frac{(\theta_{j,n}(u)-\theta)^{2d-1}}{(2d-1)!}\dd\theta.
 \end{equation*}
 Since the moments $\mu_1,\ldots,\mu_{2d-2}$ that do not depend on $u$
 but $\mu_{2d-1}=u$, substracting \eqref{FxGnu}  with $u=0$  from
 \eqref{FxGnu} yields
 \begin{equation*}
  \frac{f(x,G_n(u))-f(x,G_n(0))}{\pi_{m_0}}=\frac{u}{n^{1/2}}f^{(2d-1)}(x,\theta_{m_0})+R_{n}(x,u)-R_{n}(x,0).
 \end{equation*}
 Dividing by $f(x,G_n(0))$, recalling \eqref{Zin} and setting 
 \begin{equation*}
  R_{i,n}(u)=\frac{R_n(X_{i,n},u)}{f(X_{i,n}, G_n(0))}, 
 \end{equation*}
 we see that \eqref{Yinu} can be written as
 \begin{equation*}
  Y_{i,n}(u)= \pi_{m_0}\left[u n^{-1/2}Z_{i,n}+R_{i,n}(u)-R_{i,n}(0)\right].
 \end{equation*}
 Moreover, for each fixed $n$ and $u$, the  i.i.d. vectors $(Y_{i,n}(u),Z_{i,n},R_{i,n}(u)) $ are centered under $G_n(0)$. Indeed, from \eqref{Yinu}, we have 
 \[\E_{G_n(0)} Y_{i,n}(u)=\int [f(x,G_n(u))-f(x,G_n(0))]\dd\lambda(x)=0;\]
 furthermore by expanding $f$ around $\theta_{m_0}$, dividing by
 $f(\cdot,G_n(0))$, taking expectations and applying Fubini Theorem to the
 remainder, we get 
 \begin{align*}
 0=\sum_{k=1}^{2d-1}\frac{h^k}{k!}\E_{G_n(0)}&
\left[\frac{f^{(k)}(X_{i,n},\theta_{m_0})}{f(X_{i,n},G_n(0))}\right]\\
&+\int_{\theta_{m_0}}^{\theta_{m_0}+h}\frac{(\theta_{m_0}+h-\theta)^{2d-1}}{(2d-1)!}\E_{G_n(0)}
\left[\frac{f^{(2d)}(X_{i,n},\theta)}{f(X_{i,n},G_n(0))}\right]\dd \theta;
 \end{align*}
 Proposition~\ref{tversmix} ensures that each expectation exists, that
 Fubini Theorem is valid  and that the remainder term is of order
 $h^{2d}$. Thus, we deduce iteratively 
 that for $k\in\lb1,2d-1\rb$
\[\E_{G_n(0)} \left[\frac{f^{(k)}(X_{i,n},\theta_{m_0})}{f(X_{i,n},G_n(0))}\right]=0\]
 and in particular $\E_{G_n(0)} Z_{i,n}=0$. And dividing \eqref{FxGnu} by
 $f(x,G_n(0))$ gives as a result $\E_{G_n(0)} R_{i,n}(u)=0$ for all $u$. 

\subsection*{$L^p$-convergences of $A_n(u)$, $B_n(u)$ and $C_n(u)$}

We show the convergences \eqref{p1}, \eqref{p2} and \eqref{p3}:
\begin{eqnarray*}
A_n(u)=\sum_{i=1}^nY_{i,n}(u) -uZ_n & \xrightarrow[]{L^2} & 0, \\
B_n(u)=  \sum_{i=1}^nY_{i,n}(u)^2-u^2\Gamma& \xrightarrow[]{L^1}  &0, \\
 C_n(u)=\sum_{i=1}^n|Y_{i,n}(u)|^3&  \xrightarrow[]{L^1}& 0. 
\end{eqnarray*}
Recall the quantities: 
\begin{eqnarray}
  \label{Rnxu}
 R_n(x,u) &=&\sum_{j=m_0}^m\pi_{j,n}(u) \int_{\theta_{m_0}}^{\theta_{j,n}(u)}f^{(2d)}(x,\theta)\frac{(\theta_{j,n}(u)-\theta)^{2d-1}}{(2d-1)!}\dd\theta,\\
\label{Yinubis}
 Y_{i,n}(u)&=& \pi_{m_0}\left[u n^{-1/2}Z_{i,n}+R_{i,n}(u)-R_{i,n}(0)\right],\\
\label{Zn_Supp}
  Z_n&=&\pi_{m_0}n^{-1/2}\sum_{i=1}^nZ_{i,n}.
\end{eqnarray}
Recall also in the following computations that for each fixed $n$ and $u$, the  i.i.d. vectors $(Y_{i,n}(u),Z_{i,n},R_{i,n}(u)) $ are centered under $G_n(0)$. 

\underline{Proof of \eqref{p1}}. Note that from \eqref{Yinubis} and \eqref{Zn}
\[A_n(u)=
\pi_{m_0}\left(\sum_{i=1}^n R_{i,n}(u)-\sum_{i=1}^n R_{i,n}(0)\right),\]
and the equalities
\[\E_{G_n(0)}\left|\sum_{i=1}^n R_{i,n}(u)\right|^2=\sum_{i=1}^n\E_{G_n(0)} R_{i,n}(u)^2= n\E_{G_n(0)}|R_{1,n}(u)|^2\]
will give the desired $L^2$-convergence  if we can prove that for each $u$,
\begin{equation}
  \label{cvR1nu}
  \E_{G_n(0)}|R_{1,n}(u)|^2=o\left(\frac1n\right).
\end{equation}
To this end, we  look at the
expression \eqref{Rnxu} of $R_{n}(x,u)$ for fixed $u$. We have
$|\theta _{j,n}(u) - \theta|^{2d - 1} \le H(u)^{2d-1} n^{-1/2}$ for any
$\theta$ in the integrand, any $j$ and $n$. We may thus write 
\begin{align*}
\left|R_n(x,u)\right| & \le \sum_{j=m_0}^m \pi_j(u) \int_{\theta _{m_0} - H(u) n^{-\frac{1}{4d-2}}}^{\theta _{m_0} + H(u)  n^{-\frac{1}{4d-2}}} \left\lvert f^{(2d)}(x, \theta) \right\rvert \frac{H(u)^{2d-1}n^{-1/2}}{(2d-1)!} \mathrm{d}\theta \\
                      & \underset{u}{\lec} n^{-1/2}  \int_{\theta _{m_0} - H(u) n^{-\frac{1}{4d-2}}}^{\theta _{m_0} + H(u)  n^{-\frac{1}{4d-2}}}  \left\lvert f^{(2d)}(x, \theta) \right\rvert \mathrm{d}\theta.
\end{align*}
Since we have $\sigma $-finite
measures, we may use Fubini theorem. Since moreover $\theta$ in the
integrand is between $\theta_{m_0}$ and $\theta _{j,n}(u)$ which
converges to $\theta_{m_0}$, we may then apply
Proposition~\ref{tversmix}. For $q\in \lb 1,4\rb$, using convexity of $x \mapsto x^q$ on line two, we may then write:
\begin{align}
  \notag  \E_{G_n(0)}\left|R_{1,n}(u)\right|^q    & \underset{u}{\lec}   n^{-q/2}  \E_{G_n(0)}\left|\frac{  \int_{|\theta-\theta _{m_0} |\underset{u}{\lec} n^{-\frac{1}{4d-2}}}   \left\lvert f^{(2d)}(\cdot, \theta) \right\rvert \dd\theta }{f(\cdot, G_n(0))}\right|^q \\
 \notag     & \underset{u}{\lec}   n^{-\frac{q}{2}-\frac{q-1}{4d-2}}
        \int_{|\theta-\theta _{m_0}|\underset{u}{\lec}
        n^{-\frac{1}{4d-2}}} \underbrace{\E_{G_n(0)}\left|\frac{f^{(2d)}(\cdot,
        \theta)  }{f(\cdot, G_n(0))}\right|^q }_{\lec 1}\dd\theta \\
\label{cvEn0}      & \underset{u}{\lec} n^{-\frac{q}{2}-\frac{q}{4d-2}}  
\end{align}
Take $q=2$ to obtain \eqref{cvR1nu}  ; the proof of \eqref{p1} is complete.

\underline{Proof of \eqref{p2}}. Write 
\[B_n(u)=B^1_n(u)+B^2_n(u),\]
with
\begin{eqnarray*}
  B^1_n(u)&=&\sum_{i=1}^nY_{i,n}(u)^2-\frac{u^2\pi_{m_0}^2}{n}\sum_{i=1}^nZ_{i,n}^2,\\
B^2_n(u) &=& \frac{u^2\pi_{m_0}^2}{n}\sum_{i=1}^nZ_{i,n}^2-u^2\Gamma.
\end{eqnarray*}
Note first that from \eqref{Yinubis}  and
\eqref{Zn_Supp},
\[ B^1_n(u)=\pi_{m_0}^2\sum_{i=1}^n(R_{i,n}(u)-R_{i,n}(0))^2+\frac{2u \pi_{m_0}^2}{\sqrt{n}}\sum_{i=1}^n(R_{i,n}(u)-R_{i,n}(0))Z_{i,n},\]
so that taking the $L^1$-norm and by the Cauchy-Schwarz inequality,
\begin{align*}
  \E_{G_n(0)}\left|B^1_n(u)\right|\underset{u}{\lec}
                                                                      n\E_{G_n(0)}&\left[|R_{1,n}(u)|^2+|R_{1,n}(0)|^2\right]\\
&+\sqrt{n\E_{G_n(0)} \left[|R_{1,n}(u)|^2+ |R_{1,n}(0)|^2\right]}\sqrt{\E_{G_n(0)} Z_{1,n}^2}
\end{align*}
and the r.h.s. tends to $0$ by \eqref{cvR1nu} and the fact that
$\E_{G_n(0)} Z_{1,n}^2\to\sigma^2$. Besides, setting $\delta_n=|\pi_{m_0}^2\E_{G_n(0)} Z_{1,n}^2-\Gamma|$, we have 
\begin{eqnarray*}
  \E_{G_n(0)}\left|B^2_n(u)\right|^2&\underset{u}{\lec} &\E_{G_n(0)}
                                                                  \left|n^{-1}\sum_{i=1}^n(Z_{i,n}^2-\E_{G_n(0)}
                                                                  Z_{1,n}^2)\right|^2+\delta_n^2\\
&\underset{u}{\lec} &
                                                                  n^{-1}\mathrm{Var}_{G_n(0)}(Z_{1,n}^2)+\delta_n^2
\end{eqnarray*}
which goes to zero since $\delta_n\to 0$ by definition and $\E_{G_n(0)}
Z_{1,n}^4\lec 1$ by Proposition~\ref{tversmix}.
We have thus,   
\[\E_{G_n(0)}\left|B_n(u)\right|\lec \E_{G_n(0)}\left|B^1_n(u)\right|+\sqrt{\E_{G_n(0)}\left|B^2_n(u)\right|^2}\xrightarrow[]{} 0\]
which proves \eqref{p2}.

\underline{Proof of \eqref{p3}}. It is easily seen from \eqref{Yinubis} that
\begin{equation*}
 C_n(u)\underset{u}{\lec}
  n^{-3/2}\sum_{i=1}^n|Z_{i,n}|^3+\sum_{i=1}^n|R_{i,n}(u)|^3+\sum_{i=1}^n|R_{i,n}(0)|^3
\end{equation*}
so that taking expectations
\[
\E_{G_n(0)} |C_n(u)|\underset{u}{\lec}
  n^{-1/2}\E_{G_n(0)}|Z_{1,n}|^3+n\E_{G_n(0)}\left[|R_{1,n}(u)|^3+|R_{1,n}(0)|^3\right].
\]
But each of the three terms in the r.h.s. tends to $0$: the first one
because of $\E_{G_n(0)}|Z_{1,n}|^3\lec 1$ by Proposition~\ref{tversmix}, the
second and the third ones because of  \eqref{cvEn0} for $q=3$. Thus
$C_n(u)$ converges to $0$ in $L^1$.

\section{Proof of Theorem~\MakeLowercase{\ref{equal_compo}}}
Assume without loss of generality that $G_0 = \sum_{i=1}^m \pi_{i,0} \delta_{\theta_{i,0}}$, with $\pi_{i,0} \ge \pi_{\min,0} > 0$ and $\theta _{i+1,0} -\theta_{i,0}\ge  \kappa_0$ for all $i$, with $\kappa_0 > 0$. 
    
    Then, with $\eps>0$ small enough, any mixture $G$ in $\Glm \cap \mathcal{W}_{G_0}(\varepsilon)$ must have exactly one component close to each $\theta_{i,0}$, with a weight of order one. More precisely, for $\eps= \pi_{\min,0} \kappa_0 / 4$,
    \begin{equation}
      \label{eq:approx}
G \in  \Glm \cap \mathcal{W}_{G_0}(\varepsilon)\implies  \left\{
  \begin{array}{l}
 \displaystyle G = \sum_{i=1}^m \pi_i \delta_{\theta_i}, \\ 
 \text{with }\pi_i \ge \dfrac{\pi_{\min,0}}{2} \text{ and }
\left\lvert \theta_i - \theta _{i,0} \right\rvert \le \dfrac{\kappa_0}{2}.
  \end{array}\right.
    \end{equation}
Indeed, by the very definition of $W$ (not the dual form), there is a probability measure $\pi(\cdot, \cdot)$ on $\Theta \times \Theta$ with marginals $G_0=\pi(\cdot,\Theta)$ and $G=\pi(\Theta,\cdot)$ such that 
    \begin{equation}
      \label{eq:transp}
      W(G,G_0)=\sum_{i,j=1}^m|\theta_{i,0}-\theta_j|\pi\left(\{\theta_{i,0}\},\{\theta_j\}\right).
    \end{equation}
Set  $J_{i,0}=\{j: |\theta_{i,0}-\theta_j|<\kappa_0/2\}$ for each $i\in \lb 1,m\rb$. Then, from \eqref{eq:transp}, for each $i$,
\begin{eqnarray*}
   W(G,G_0) &\ge & \sum_{j\notin J_{i,0}}|\theta_{i,0}-\theta_j|\pi\left(\{\theta_{i,0}\},\{\theta_j\}\right)\\
&\ge &\frac{\kappa_0}{2}\sum_{j\notin J_{i,0}}\pi\left(\{\theta_{i,0}\},\{\theta_j\}\right)\\
&= &\frac{\kappa_0}{2}\left[\pi_{i,0}-\sum_{j\in J_{i,0}}\pi\left(\{\theta_{i,0}\},\{\theta_j\}\right)\right]\\
&\ge &\frac{\kappa_0}{2}\left[\pi_{\min,0}-\sum_{j\in J_{i,0}}\pi\left(\{\theta_{i,0}\},\{\theta_j\}\right)\right].
\end{eqnarray*}
Thus, if $W(G,G_0)<\pi_{\min,0} \kappa_0 / 4$, then we must have for each $i$,
\begin{equation}
  \label{eq:min}
\sum_{j\in J_{i,0}}\pi\left(\{\theta_{i,0}\},\{\theta_j\}\right)>\frac{\pi_{\min,0}}{2}
\end{equation}
and each $J_{i,0}$ is non empty. Furthermore, the (disjoint) $J_{i,0}$'s, $i\in \lb 1,m\rb$, are all singletons ; otherwise there would be at least one $J_{i,0}$ empty, since $G_0$ has exactly $m$ support points and $G$ at most $m$.  Considering a suitable numbering for the components of $G$, we can thus write $J_{i,0}=\{i\}$  so that $|\theta_{i,0}-\theta_i|<\kappa_0/2$ for each $i$ and \eqref{eq:min} yields
$\pi_i\ge \pi\left(\{\theta_{i,0}\},\{\theta_i\}\right)>\pi_{\min,0}/2$.

Now, set 
\[L=\inf_{\substack{G_1,G_2\in\Glm \cap \mathcal{W}_{G_0}(\eps)\\G_1\ne G_2}}   \frac{\left\lVert F(\cdot, G_1) - F(\cdot, G_2) \right\rVert_\infty}{W(G_1, G_2)}.\]
Select sequences of  mixing distributions $G_{1,n}\ne G_{2,n}$ in  $\Glm \cap \mathcal{W}_{G_0}(\varepsilon)$ such that:
    \begin{align*}
        \frac{\left\lVert F(\cdot, G_{1,n}) - F(\cdot, G_{2,n}) \right\rVert _\infty}{W(G_{1,n}, G_{2,n})} &\xrightarrow[n\to \infty]{} L.
    \end{align*}
We have to prove that $L>0$. Actually we shall prove that for $n$
large enough,
\begin{equation}
  \label{min1}
\frac{\left\lVert F(\cdot, G_{1,n}) - F(\cdot, G_{2,n}) \right\rVert _\infty}{W(G_{1,n}, G_{2,n})}\gec 1. 
\end{equation}
Up to taking subsequences, we can write $G_{a,n} =  \sum_{j=1}^m \pi_{j,a,n} \theta _{j,a,n}$
    with the convergences $\pi_{j,a,n} \to
    \pi_{j,a,\infty}$ and $\theta_{j,a,n} \to \theta
    _{j,a,\infty}$, for $a \in \{1,2\}$. Note that $G_{a,\infty} =  \sum_{j=1}^m
    \pi_{j,a,\infty} \delta_{\theta _{j,a,\infty}}$ lies in $\Glm \cap
    \mathcal{W}_{G_0}(\varepsilon)$ and thus satisfies
    \eqref{eq:approx}. In particular, the $\theta _{j,a,\infty}$'s are
    $\eps_a$-separated for some $\eps_a>0$ ; this will be used for
    $a=2$ below. 

Note first that for any $1$-Lipschitz function $f$ and any mixing distributions $G,G'\in \Glm$,
\[\left|\int_{\Theta}f\dd (G-G')\right|\le \sum_{j=1}^m \left| \theta_{j} - \theta'_{j} \right| +\Diam(\Theta) \left|  \pi_{j} - \pi_{j}'\right|\]
so that 
\[W(G_{1,n}, G_{2,n})\lec \sum_{j=1}^m \left| \theta_{j,1,n} - \theta_{j,2,n} \right| + \left|  \pi_{j,1,n} - \pi_{j,2,n}\right|.\]
To obtain \eqref{min1}, it remains to prove that for large $n$,
\begin{equation}
  \label{minum}
  \left\lVert F(\cdot, G_{1,n}) - F(\cdot, G_{2,n}) \right\rVert _\infty\gec \sum_{j=1}^m \left| \theta_{j,1,n} - \theta_{j,2,n} \right| + \left|  \pi_{j,1,n} - \pi_{j,2,n}\right|.
\end{equation}
By Taylor expansion of $F(x,\theta_{j,1,n})$ around $\theta_{j,2,n} $ and
Assumption B(1), 
\begin{equation}
  \label{DL}
F(x,G_{1,n}) - F(x,G_{2,n})  =\Sigma_n(x)+o\left(\sum_{j=1}^m
  |\theta _{j,1,n} - \theta _{j,2,n}|\right),\end{equation}
with 
\[\Sigma_n(x)=\sum_{j=1}^m  (\pi _{j,1,n} - \pi _{j,2,n}) F(x, \theta
_{j,2,n} )+\pi_{j,1,n} (\theta_{j,1,n} - \theta_{j,2,n}) F'(x,
\theta_{j,2,n} ).\]
In addition,  by convergence of $\theta_{j,a,n}$ to
$\theta_{j,a,\infty}$ for each $j$ and $a=1,2$, there is an integer $n_0$ such that for all $n\ge
n_0$,  each $(\theta_{j,2,n})_{1\le
  j\le m}$ is $\dfrac{\eps_2}{2}$-separated and $\pi_{j,1,n}\ge \pi_{j,1,\infty} /2$ for each $j$. So that by
Proposition~\ref{infimum}, for all $n\ge n_0$,
 \[\|\Sigma_n(\cdot)\|_\infty\gec \sum_{j=1}^m  |\pi _{j,1,n} - \pi _{j,2,n}|+\frac{\pi_{j,1,\infty}}{2} |\theta_{j,1,n} - \theta_{j,2,n}|. \]
Since \eqref{eq:approx} holds for  $G_{1,\infty}=\sum_{j=1}^m
\pi_{j,1,\infty}\delta_{\theta_{j,1,\infty}}$, we have $\pi_{j,1,\infty}\ge \pi_{\min,0}/2$ for each $j$. Thus, for all $n\ge n_0$,
\[\|\Sigma_n(\cdot)\|_\infty\gec \sum_{j=1}^m  |\pi _{j,1,n} - \pi
_{j,2,n}|+|\theta_{j,1,n} - \theta_{j,2,n}|.\]
Combining this last inequality with the sup-norm of \eqref{DL} gives
\eqref{minum}. This ends the proof.

\end{document}